\documentclass[11pt]{amsart}
\usepackage{amsmath, amsthm, amssymb,amscd}
\usepackage{float}
\usepackage{graphicx}
\usepackage[arrow, matrix, curve]{xy}
\usepackage{color}
\usepackage{enumitem}  
\usepackage{tikz-cd} 
\usepackage{url}

\usepackage{blindtext}
\usepackage[english]{babel}

\newcommand{\proj}{\mathbf{P}}

\newcommand{\seq}{\subseteq}
\newcommand{\C}{\mathbb{C}}

\newtheorem{thm}{Theorem}[section]
\newtheorem*{thm-nl}{Theorem}
\newtheorem*{prop-nl}{Proposition}
\newtheorem{definition}[thm]{Definition}
\newtheorem{lem}[thm]{Lemma}

\def\PP{{\textbf P}}
\def\OO{\mathcal{O}}

\def\cM{\mathcal{M}}

\def\H{\mathcal{H}}
\def\Pic0{{\rm Pic}^0(X)}

\def\mm{\overline{\mathcal{M}}}

\newtheorem{cor}[thm]{Corollary}
\newtheorem*{cor-nl}{Corollary}

\newtheorem*{conjecture-nl}{Conjecture}
\newtheorem{defin}[thm]{Definition}
\newtheorem*{quest-nl}{Question}
\newtheorem*{quests-nl}{Questions}

\newtheorem{prop}[thm]{Proposition}

\theoremstyle{remark}

\title{ {Betti Numbers of Curves and Multiple-Point Loci}}
\date{\today}

\author[M. Kemeny]{Michael Kemeny}

\address{University of Wisconsin-Madison, Department of Mathematics, 480 Lincoln Dr
\hfill \newline\texttt{}
 \indent WI 53706, USA} \email{{\tt michael.kemeny@gmail.com}}

\begin{document}
\begin{abstract}
We construct Eagon--Northcott cycles on Hurwitz space and compare their classes to Kleiman's multiple point loci. Applying this construction towards the classification of Betti tables of canonical curves, we find that the value of the extremal Betti number records the number of minimal pencils. The result holds under transversality hypotheses equivalent to the virtual cycles having a geometric interpretation. We analyse the case of two minimal pencils, showing that the transversality hypotheses hold generically. 
\end{abstract}
\maketitle
\setcounter{section}{-1}
\section{Introduction}
The canonical ring $\Gamma_C(\omega_C):= \bigoplus_{m \in \mathbb{N}} H^0(C,\omega_C^{\otimes m})$ of a smooth curve $C$ is of central interest in algebraic geometry. To describe its structure, define
$$b_{i,j}(C,\omega_C):= \dim \mathrm{K}_{i,j}(C,\omega_C),$$
where $\mathrm{K}_{i,j}(C,\omega_C):=\mathrm{Tor}_S^i(\Gamma_C(\omega_C),\C)_{i+j}$ and where $S$ denotes the polynomial algebra \newline $\text{Sym}(H^0(C,\omega_C))$. 
\smallskip

The theorem of Noether--Babbage--Petri states that, for any curve $C$ of gonality at least $4$ which is not a plane quintic, we have $b_{1,q}(C,\omega_C)=b_{0,q}(C,\omega_C)=0$ for $q \geq 2$. Furthermore, the number of quadrics required to cut out $C$ is given by $\displaystyle{b_{1,1}(C,\omega_C)}.$ \smallskip


It is natural to ask what geometric information is contained within the remaining Betti numbers. Thanks to the combined work of many authors, one knows precisely which of the $b_{i,j}:=b_{i,j}(C,\omega_C)$ are nonzero for a curve of gonality $k$ under a specific generality hypothesis, \cite{V1}, \cite{V2}, \cite{aprodu-remarks}. By Koszul duality \cite{green-koszul}, it suffices to consider the \emph{linear} Betti numbers $$b_{1,1}, b_{2,1}, \ldots , b_{p,1}, \ldots.$$ Suppose that the curve $C$ has gonality $k$ and satisfies a \emph{linear growth} condition on the dimension of the moduli space of linear series of projective dimension one,  \cite{aprodu-remarks}.
Then $b_{p,1} \neq 0$ if and only if $1 \leq p \leq g-k$, \cite{aprodu-remarks}. \smallskip

Little has been known regarding the question of how the values of the nonzero entries $b_{p,1}$ reflect the geometry of the curve. In genus nine or less, a classification of the possible Betti tables is available, \cite{schreyer1}, \cite{sagraloff}. Conjectural tables for genus $10$ and $11$ have been produced via computer experiment, \cite{schreyer-topics}. \smallskip

The extremal Betti number $b_{g-k,1}$ has been studied intensely under the assumptions that the curve is a general point in the moduli space of $k$-gonal curves, \cite{lin-syz}, \cite{projecting}, \cite{kemeny-voisin}. In this paper, we study \emph{special} curves within the locus of $k$-gonal curves. Namely, we consider curves which carry \emph{multiple} pencils of degree $k$. \smallskip

 The extremal Betti number $b_{g-k,1}$ stands out from the tables in \cite{schreyer-topics} as the quantity responsible for most of the variance in the Betti tables. Furthermore,  $b_{g-k,1}$ has a remarkably close relationship to the geometry of the curve $C$. Schreyer's experiments suggest that, if we make the assumption that $C$ carries only finitely many minimal pencils, then $$\boxed{\; \; b_{g-k,1}(C,\omega_C)=m(g-k) \; \;}$$ 
where $m$ counts minimal pencils of $C$, i.e.\ degree $k$ maps $C \to \PP^1$, with multiplicity. The above formula is an incarnation of the philosophy that syzygies of canonical curves tend to arise from special linear systems, as stated in the \emph{Geometric Syzygy Conjecture} \cite{vB1}, which is now proven for \emph{general} curves, \cite{kemeny-rank}. \smallskip

The goal of this paper is to provide an explanation for Schreyer's experimental observation. Certain exceptions, however, to the above formula are apparent. If $C$ is a smooth plane sextic, then $g=10$ and the extremal Betti number is $b_{6,1}(C,\omega_C)=27$, which is not even divisible by $6$ (and, further $k=5$, so the linear strand has the wrong length). If $C$ is a genus $11$ curve admitting a degree three cover of an elliptic curve, then experiments suggest $b_{5,1}(C,\omega_C)=27$. One major difficulty in proving a formula along the lines of Schreyer's empirical observation is that one must first come up with appropriate hypotheses to rule out such exceptions.\smallskip
 
One major objective of this paper is to work out such a precise transversality assumption. Guided by Aprodu's results on specific Green's Conjecture \cite{aprodu-remarks}, and following \cite{lin-syz}, a curve $C$ of genus $g$ and gonality $k\leq \frac{g+1}{2}$ satisfies \emph{bpf-linear growth} provided we have the dimension estimates on the dimensions of Brill--Noether loci:
\begin{align*}
\dim G^1_{k+m}(C) &\leq m,  \;\; \text{for $0 \leq m \leq g-2k+1$} \\
\dim G^{1,\mathrm{bpf}}_{k+m}(C) &<m, \;\; \text{for $0 < m \leq g-2k+1$} .\end{align*} 
\smallskip

 The bpf-linear growth condition appears in well-known works of Martens--Mumford and Keem on the dimensions of Brill--Noether loci, \cite[Ch.\ IV]{ACGH1}. For low $k$, curves violating bpf-linear growth tend to be either plane curves or low degree covers of curves of low genus. Work of Aprodu--Farkas \cite{aprodu-farkas} implies that if $C$ is a curve of non-maximal gonality which can be abstractly embedded on a K3 surface and if, furthermore, the only line bundles computing the Clifford index are minimal pencils, then $C$ satisfies bpf-linear growth. See the Appendix for more on the bpf-linear growth condition.  \smallskip

We may now state our first main result:
\begin{thm} \label{geo-thm}
 Let $C$ be a smooth curve of genus $g$ and non-maximal gonality $k \leq \lfloor \frac{g+1}{2} \rfloor$, satisfying bpf-linear growth. Assume the minimal pencils are in general position and have ordinary ramification. Then 
$$b_{g-k,1}(C,K_C)=m(g-k),$$ 
where $m=\# W^1_k(C)$. 
\end{thm}
See Definition \ref{gen-position-def} for the precise definition of general position. Roughly speaking, we assume that there are no obstructions to deforming the minimal pencils, and further that there are no global relations amongst these pencils.

Furthermore, we prove that, under the hypothesis of Theorem \ref{geo-thm}, all extremal linear syzygies arise from scrolls. i.e.\ there is a natural isomorphism
$$\bigoplus_{i=1}^m K_{g-k,1}(X_{f_i}, \mathcal{O}_{X_{f_i}}(1)) \simeq K_{g-k,1}(C,\omega_C)$$
where $X_{f_1}, \ldots, X_{f_m}$ are the scrolls associated to the minimal pencils, \cite[\S 2]{schreyer1}.
The syzygies of the scrolls $X_{f_i}$ are explicitly described via the Eagon--Northcott complex, see \cite[\S 0.2]{lin-syz}. The above theorem thus provides an \emph{explicit} description of the syzygy space $ K_{g-k,1}(C,\omega_C)$.  \smallskip

The assumption in Theorem \ref{geo-thm} that the pencils be in general position is required. For instance, the conclusion fails if $m=1$ but $W^1_k(C)$ is not reduced, \cite[Prop.\ 10]{SSW}. The conclusion further fails for a general curve of even genus and maximal gonality, in which case the minimal pencils fail to be in general position.\smallskip 


It is a difficult problem to specify precise conditions under there exist genus $g$ curves admitting precisely $m$ minimal pencils. If $m=1$, such curves always exist for $k \leq \lfloor \frac{g+1}{2} \rfloor$, \cite{arbarello-cornalba}. For, $m=2$, there exist genus $g \geq 8$ curves with two mutally independent pencils if and only if $g > (k-1)^2$, $g > (k-1)^2$, \cite{jongmans}, \cite{coppens}. Furthermore, the moduli space of curves with two independent minimal pencils is irreducible, \cite{Ty}. We show:
\begin{thm}
Let $g \geq 8$, $k \geq 6$ and $g>(k-1)^2$. Then there exist smooth curves of genus $g$ with precisely two minimal pencils satisfying the assumptions of Theorem \ref{geo-thm}. In particular, $b_{g-k,1}(C,K_C)=2(g-k)$ for such curves.
\end{thm}

Less is known for $m\geq 3$. Based on computer experiments one expects the existence of curves of genus $11$ and gonality $6$ with $m$ pencils in general position for $1 \leq m \leq 10$, \cite{schreyer-topics}, \cite{bopp-schreyer}. 



\subsection{Method of Proof} 
 A variational approach allows one to study the Betti numbers of curves $C$ via the geometry of moduli spaces, see \cite{hirsch}, \cite{generic-secant}. Whilst previous constructions provide syzygy divisors on the moduli space of curves, we must adapt this approach to work with cycles of higher codimension. \smallskip

Our inspiration comes from Herbert's multiple point formula, \cite{kleiman}, \cite[Example 9.1.14]{fulton}. Let $f: X \to Y$ be an unramified, proper morphism of smooth varieties. For fixed $m$, Herbert's formula computes the class of the loci of those $y \in Y$ with $$\#f^{-1}(y)\geq m,$$ under the assumption that $f$ is \emph{self-transverse}, i.e.\ that $T_{x_1}(X),\ldots,T_{x_m}(X)$ are in general position in $T_y(Y)$, for $\{x_1, \ldots, x_m\}=f^{-1}(y)$. \smallskip

Consider the moduli space $\mm_{g,k}(\PP^1,\{0,1,\infty\})$ of stable maps of genus $g$ and degree $k$ to $\PP^1$, with fixed base points over $0,1,\infty$. Denote by $\cM_{g,k}(\PP^1,\{0,1,\infty\})$ the unique irreducible component such that the general point is a morphism with \emph{smooth} base $C$. There is a generic immersion $\displaystyle{\pi: \mathcal{M}_{g,k}(\PP^1,\{0,1,\infty\}) \to \mm_{g,3}},$
 defined by sending a marked stable map to its base. \smallskip
 
We first focus on the divisorial case $g=2k-1$. Consider the space $\cM^{\text{o}}_{2k-1,3}$ of irreducible, automorphism-free curves with three marked points and let $\H(0) \seq \cM^{\text{o}}_{2k-1,3}$ be the locus such that $\pi$ is self-transverse over $\H(0)$. Set $\H(1)=\pi^{-1}(\H(0))$. We firstly construct a virtual cycle $\mathcal{EN}_m$ in the Chow group $A^m(\H(1))$. Under a transversality assumption, $\mathcal{EN}_m$ represents the locus 
$$ \{ C \to \PP^1 \in \H(1) \; | \; b_{g-k,1}(C,\omega_{C}) >m(g-k)  \}.$$\smallskip

We further define Brill--Noether cycles $\mathcal{BN}_{m+1}$ parametrizing curves carrying at least $m+1$ pencils as multiple point loci, under appropriate transversality hypotheses, see Definition \ref{BN-cyc-defin}.
We prove:
\begin{thm} \label{cyc-comp}
We have the following equality of virtual cycles in $A^m(\H(1))$
$$\mathcal{EN}_{m}=(k-1)\mathcal{BN}_{m+1}.$$
\end{thm}
This provides an intersection--theoretic explanation for the experimental observation \newline $b_{g-k,1}(C,\omega_C)=m(g-k)$ with $m=\# W^1_k(C)$, in the special case $g=2k-1$.\smallskip

To upgrade this virtual computation into a geometric statement, we require that the pencils be in general position. Let $C$ be a smooth curve of gonality $k$ with finitely many minimal pencils $f_1, \ldots, f_m : C \to \PP^1$, all of type $I$.\footnote{A line bundle $L$ with $h^0(L)=2$ is said to be of type I if $W^1_k(C)$ is smooth and zero dimensional at $[L]$. This is equivalent to having $h^0(L^2)=3$.} Choose general points $p,q,r \in C$. The pencils are said to be \emph{infinitesimally in general position} if, for all subsets $\sigma=\{\sigma_1, \ldots, \sigma_j\} \seq \{1, \ldots, m \}$, $$H^1\left(C,N_{F_{\sigma}}(-p-q-r)\right)=0,$$
where $F_{\sigma}: C \to (\PP^1)^{| \sigma|}$ is the map with $i^{th}$ projection given by $f_{\sigma_i}$ and $N_{F_{\sigma}}$ denotes the normal sheaf. This condition states that the deformation theory of the collection of pencils is unobstructed and appears in work of Arbarello--Cornalba, \cite{arbarello-cornalba}.


In order for the Eagon--Northcott cycles to carry their natural meaning, we further impose a condition which is \emph{global} in nature. Fix a general effective divisor $T$ of degree $g-1-k$ general points on $C$. To each minimal pencil $f_i$ of type I one naturally associates a rank $4$ quadric $Q_i \seq \PP^{g-1}$ following \cite{green-quadrics}, see Section \ref{GGP}. We say that $f_1, \ldots f_m$ are in \emph{geometrically general position} if there are no linear relations amongst the associated quadrics $\{Q_1, \ldots, Q_m\} \seq |\mathcal{O}_{\proj^{g-1}}(2)|$.

\begin{definition} \label{gen-position-def}
Let $C$ be a smooth curve of gonality $k$. We say that the locus $W^1_k(C)$ of minimal pencils is in \textbf{general position} if all elements are pencils of type $I$ and the pencils are both infinitesimally and geometrically in general position. 
\end{definition}
The infinitesimal general position condition ensures that the Brill--Noether cycles $\mathcal{BN}_{m+1}$ carry geometric meaning, whereas geometrically general position ensures that the Eagon--Northcott cycles $\mathcal{EN}_{m}$ carry geometric meaning.\smallskip

We end the introduction with a sketch of the proof of Theorem \ref{geo-thm}. Let $C$ be a curve of genus $g$ and gonality $k$ with $m$ minimal pencils. We view the curve $C$ with the $m$ pencils as providing us with a stable map $C \to (\PP^1)^m$. We wish to move ourself into the setting where the cycle computations of Theorem \ref{cyc-comp} apply, i.e.\ into a setting with $g=2k-1$. To this end, we set $n=g+1-2k$ and construct a stable map $D \to (\PP^1)^m$ from a nodal curve $D$ of arithmetic genus $g+n$, with each factor $D \to \PP^1$ being a map of degree $k+n$, see Section \ref{sect=key-const}. We now have $g(D)=2(k+n)-1$, so may hope to apply a degenerate version of Theorem \ref{cyc-comp}, under the transversality hypotheses that the Brill--Noether and Eagon--Northcott cycles carry geometric meaning. See Section \ref{outline} for an outline with more details provided.

We remark that we expect some version of Theorem \ref{geo-thm} to hold under weaker hypotheses, for instance with the pencils not necessarily in infinitesimally general position, and with $m$ counting the number of pencils \emph{with multiplicity}, i.e.\ by letting $m$ denote the length of the zero dimensional Brill--Noether variety $W^1_k(C)$. We are unable to prove such a result using our methods, however, as they make essential use of Herbert's multiple point formula, which requires a heavy transversality assumption for its validity. In order to satisfy this assumption, we are forced to require $W^1_k(C)$ to be reduced. \smallskip


\textbf{Acknowledgements} I thank the referees for careful readings. We thank C.\ Bopp, D.\ Eisenbud, G.\ Farkas, H.\ Keneshlou and F.\ Schreyer for discussions. Thanks to R.\ Yang for comments on a draft of this paper. This work was partially supported by NSF grant DMS-1701245.

\section*{Glossary of Moduli Spaces}
\begin{description}
\item [$ \mm_{g,k}((\PP^1)^{ m}, \{0,1,\infty \})$] The moduli space of genus $g$ stable maps to $(\PP^1)^m$ in the class $k [\Delta]$, where $\Delta$ denotes the small diagonal, and with three base points over $(\alpha,\ldots,\alpha)$ for $\alpha \in \{0,1,\infty \}$. \hfill \\ 

\item [$ \cM^{ns}_{g,k}((\PP^1)^{ m}, \{0,1,\infty \})$] The open substack of $ \mm_{g,k}((\PP^1)^{ m}, \{0,1,\infty \})$ parametrising morphisms $f: C \to (\PP^1)^m$ such that $C$ has non-separating nodes and further $f_i$ is finite with $h^0(f^*_i \mathcal{O}_{\PP^1}(1))=2$, for each factor $f_i$ of $f$. Further, we demand that $f_i$ is etale near the base points $(p,q,r) \in C$ for all $1 \leq i \leq m$. \hfill \\ 

\item [$ \cM_{g,k}((\PP^1)^{ m}, \{0,1,\infty \})$] The closure of $ \cM^{ns}_{g,k}((\PP^1)^{ m}, \{0,1,\infty \})$ in $ \mm_{g,k}((\PP^1)^{ m}, \{0,1,\infty \})$.\smallskip 

\noindent We warn the reader that some use similar notation for the \emph{different} stack of stable maps with smooth base.
\hfill \\ 

\item [$\pi_k: \cM_{g,k}(\PP^1, \{0,1,\infty \})\to \mm_{g,3}$] The natural forgetful morphism. \hfill \\

\item [$\widetilde{\H}(m)$] This is defined to be $\cM^{ns}_{2k-1,k}((\PP^1)^m,\{0,1,\infty\}).$ \hfill \\ 

\item [$\mathcal{H}(1)$] The largest open substack of $\cM_{g,k}(\PP^1, \{0,1,\infty \})$ such that $\pi_k$ is unramified and self-transverse on $\mathcal{H}(1)$ and, further, for any $x \in \mathcal{H}(1)$, $y \in \pi_k^{-1}(\pi_k(x))$, the base of $y$ is irreducible and automorphism-free. \hfill \\ 
\end{description}

\section{Preliminaries} \label{prelims}
\label{notation}

All Chow groups are taken with $\mathbb{Q}$ coefficients. All schemes and stacks are defined over $\C$.\smallskip

 Let $\mathcal{X}, \mathcal{Y}$ be smooth varieties over $\C$. Let $f: \mathcal{X} \to \mathcal{Y}$ be finite and unramified. We inductively define schemes $\mathcal{X}(m)$ and finite, unramified morphisms $$f(m): \mathcal{X}(m) \to \mathcal{X}(m-1).$$ Set  $\mathcal{X}(1):=\mathcal{X}$, $\mathcal{X}(0):=\mathcal{Y}$, and $f(1):=f$. Assuming we have defined $f(m-1)$, the diagonal morphism $\Delta_{f(m-1)}$ is an open immersion. Define $$ \mathcal{X}(m):= \mathcal{X}(m-1) \times_{\mathcal{X}(m-2)} \mathcal{X}(m-1) \setminus \text{Im}(\Delta_{f(m-1)}),$$ and let $f(m): \mathcal{X}(m) \to \mathcal{X}(m-1)$ be projection to the first factor.
 \begin{defin} We say $f:\mathcal{X} \to \mathcal{Y}$ as above is \textbf{self-transverse} if, for each closed point $y \in \mathcal{Y}$ and $\{z_1, \ldots, z_r \}=f^{-1}(y)$, the image of the tangent spaces $T_{z_i}(\mathcal{X})$, $1 \leq i \leq r$ under $df$ are in general position in $T_y \mathcal{Y}$, i.e.\ for all $\{\sigma_1, \ldots, \sigma_t\} \seq f^{-1}(y)$
 $$\dim T_{\sigma_1}(\mathcal{X}) \cap \ldots \cap T_{\sigma_t}(\mathcal{X})=\dim \mathcal{Y}-t(\dim \mathcal{Y}-\dim \mathcal{X}).$$
 \end{defin} 
 \smallskip

Self-transversality was stated by Herbert to ensure the validity of the Multiple Point Formula, \cite{herbert-thesis}, see also \cite[Example 9.1.14]{fulton}. If $f$ is self-transverse and $f^{-1}(y)$ has cardinality $r$ for some point $y \in \mathcal{Y}$, then $f$ satisfied Kleiman's condition of being ``$r$-generic'' of codimension $n=\dim \mathcal{Y}-\dim \mathcal{X}$, \cite[\S 4.5]{kleiman}. Moreover, each $\mathcal{X}(m)$ is nonempty and smooth of dimension $\dim \mathcal{Y}-mn$ for $m \leq r$ (so $r \leq \frac{\dim \mathcal{Y}}{n}$), \cite[Prop.\ 4.6]{kleiman}.\smallskip

The following consequence of self-transversality will be of fundamental importance.
\begin{prop} \label{self-trans-order}
  Let $f: \mathcal{X} \to \mathcal{Y}$ be a finite, unramified morphism of smooth, irreducible complex varieties. Assume $\dim \mathcal{X}=\dim \mathcal{Y}-1$ so that the image $f(\mathcal{X})$ is a divisor in $\mathcal{Y}$. Assume in addition $f$ is self-transverse. Then $$\text{ord}_y(f(\mathcal{X}))=\#f^{-1}(y),$$
  where $\text{ord}_y(f(\mathcal{X})):=\text{max} \{n \; | \; g \in I^n_y \}$ for any local holomorphic equation $g \in \hat{\mathcal{O}}_{\mathcal{Y},y}$ of $f(\mathcal{X})$.
  \end{prop}
  \begin{proof}
    Let $V_1, \ldots, V_r$ denote the tangent spaces to $\mathcal{X}$ at the points $p_1, \ldots, p_r$ over $y$. Let $U \seq \C^{m}$ for $m=\dim \mathcal{Y}$ be a small analytic neighbourhood of $y=0 \in \C^m$. Let $n_i$ denote a unit normal vector to the hyperplane $V_i\seq \C^n$ for each $i$. As the $V_i$ are in general position the $n_i$ are linearly independent, so we may assume $n_i$ is the $i$-th standard basis vector and $V_i$ is defined by $x_i=0$. As unramified morphisms are local-analytic closed immersions, we have $g=g_1\ldots g_r$ where $g_i$ defines a hypersurface with tangent plane $V_i$. Hence $g_i=c_ix_i \; \text{mod $I^2_p$}$ for nonzero constants $c_i$ and thus $g=cx_1\ldots x_r \; \text{mod $I^{r+1}_p$}$, for a nonzero constant $c$. The claim follows.
  \end{proof}

 Let $X$ be a smooth, projective, complex variety and $\beta \in H_2(X,\mathbb{Z})$. Let $P=\{ p_1, \ldots, p_{\alpha} \}$ be a collections of distinct points of $X$. For any integer $g \geq 0$ we let $\mm_{g,\beta}(X,P;n)$ denote the stack of genus $g$ stable maps in the class of $\beta$ with base point $P$ and $n$ markings, \cite[\S 10]{ara-kol}. Points of $\mm_{g,\beta}(X,P;n)$ consist of morphisms $f: C \to X$ together with markings $p'_1, \ldots, p'_{\alpha}, q_1, \ldots, q_n$ in the smooth locus of the genus $g$, connected nodal curve $C$ such that:
 \begin{enumerate}
 \item $f_*[C]=\beta$.
 \item $f(p'_i)=p_i$ for $1 \leq i \leq \alpha$.
 \item The datum $(f, p'_i,q_j)$ has finite automorphism group.
 \end{enumerate}
 
 When $n=0$ we set $\mm_{g,\beta}(X,P):=\mm_{g,\beta}(X,P;0)$. If $X=\PP^1$ we write $\mm_{g,k}(\PP^1, P;n)$ for $\mm_{g,k[\PP^1]}(\PP^1, P;n)$ and for $m\geq 2$ we write $\mm_{g,k}((\PP^1)^{ m}, P;n)$ for $\mm_{g,k[\Delta]}((\PP^1)^{m}, P;n)$ where $\Delta$ is the class of the small diagonal $\{(x,\ldots,x) \; | \; x \in \PP^1 \}$. Setting $$P=\{(0)^m,(1)^m,(\infty)^m)\}:=\{ (0,\ldots,0), (1,\ldots,1), (\infty,\ldots, \infty) \} \seq (\PP^1)^m,$$ we write $$\mm_{g,k}((\PP^1)^{ m}, \{0,1,\infty \};n)$$ for $\mm_{g,k}((\PP^1)^{ m}, \{(0)^m,(1)^m,(\infty)^m)\};n)$.\\
 
  We have a proper morphism
 $$\overline{\pi_{k}} : \mm_{g,k}(\PP^1, \{0,1,\infty \};n) \to \mm_{g,3+n}$$
 given by mapping a stable marked map to (the stabilization of) its base.
 We let $$\psi_{i}(m): \; \mm_{g,k}((\PP^1)^{ m}, \{0,1,\infty \};n) \to \mm_{g,k}((\PP^1)^{ m-1}, \{0,1,\infty \};n)$$
 for $i=1$ respectively $i=2$ be the map induced from the projection $(\PP^1)^m \to (\PP^1)^{m-1}$ away from the last respectively the first factor of $(\PP^1)^m$.\\
 
 We let $\displaystyle{\cM^{ns}_{g,k}((\PP^1)^{ m}, \{0,1,\infty \};n) \seq \mm_{g,k}((\PP^1)^{ m}, \{0,1,\infty \};n)}$
 denote the open locus parametrising marked stable maps $[f: C \to (\PP^1)^m]$ such that the base $C$ has only non-separating nodes, and further, if $f_i:=pr_i \circ f$, for $pr_i: (\PP^1)^{ m}\to \PP^1$ the $i^{th}$ projection, then $f_i$ is \emph{finite} with $h^0(C,f_i^* \mathcal{O}_{\PP^1}(1))=2$. We additionally demand that $f_i$ be \'etale near the base points $(p,q,r) \in C$ over $(0,1,\infty)$, for $1 \leq i \leq m$. \\

 Denote by $$\cM_{g,k}((\PP^1)^{ m}, \{0,1,\infty \};n) \seq \mm_{g,k}((\PP^1)^{ m}, \{0,1,\infty \};n)$$
 the closure of $\cM^{ns}_{g,k}((\PP^1)^{ m}, \{0,1,\infty \};n)$. Let $$\pi_k : \cM_{g,k}(\PP^1, \{0,1,\infty \};n) \to \mm_{g,3+n},$$
 denote the restriction of $\overline{\pi_k}$ to $\cM_{g,k}(\PP^1, \{0,1,\infty \};n)$.\\

 In the special case $g=2k-1$, set $$\widetilde{\H}(m):=\cM^{ns}_{2k-1,k}((\PP^1)^m,\{0,1,\infty\}).$$ 
 By abuse of notation, we write $\displaystyle{\psi_{i}(m): \widetilde{\H}(m) \to \widetilde{\H}(m-1)}$ for the restriction of  $\psi_{i}(m)$ to $\widetilde{\H}(m)$, $i=1,2$.\\

Let $\H(1)$ denote the largest open substack of $\cM_{2k-1,k}(\PP^1,\{0,1,\infty\})$ such that for any point $x \in \H(1)$, each point $y=[f:(C,p,q,r)] \to \PP^1 \in \widetilde{\H}(1)$ with $\pi_k(x)=\pi_k(y)$ satisfies the following conditions:
 \begin{enumerate}
   \item $C$ is irreducible, $\text{Aut}[C,p,q,r]=\{ \text{id} \}$ and $f$ is \'etale near $(p,q,r)$.
   \item $\pi_k$ is unramified and self-transverse in an open subset about $\pi_k(x)$.     
\end{enumerate}
Note that as $C$ is irreducible, $\H(1)$ is smooth of dimension $3g-1$ and $\H(1) \seq \widetilde{\H}(1).$ Self-transversality of $\pi_k$ is an open condition (cf.\ the proof of Proposition \ref{self-trans-lemma}). By definition of $\H(1)$, there is an open subset $$\H(0) \seq \mm_{2k-1,3}$$ with $\pi_k^{-1}(\H(0))\simeq\H(1)$. We continue to denote the restriction $\pi_k:\H(1) \to \H(0)$ by $\pi_k$. By the assumption $\text{Aut}[C,p,q,r]=\{ \text{id} \}$, both $\H(0)$ and $\H(1)$ are \emph{schemes}, \cite{AC1}.\\
 
 Denote by $\mathfrak{hur} \in A^1(\mm_{2k-1,3}, \mathbb{Q})$ the pullback of the Hurwitz divisor on $\mm_{2k-1}$, \cite{ha-mu}. Let $$\mathcal{A}_{g,k}=\mm_{0,2g+2k-2}(\mathcal{B} \mathfrak{S}_k)$$ be the moduli space of degree $k$ admissible covers of genus $g$, with ordered branch points. We have a natural projection $\pi_k: \mathcal{A}_{g,k} \to \mm_g$ as well as the branch morphism $$q: \mathcal{A}_{g,k} \to \mm_{2g-2k-2}.$$ Let $B_j$ denote the boundary divisors of $\mm_{0,n}$ with general point corresponding to a curve with two rational components, one of which has precisely $j$ marked points. Let $T_{\text{base}}$ be the codimension one locus in $\mathcal{A}_{g,k}$ corresponding to line bundles $l \in W^1_{k}(C)$ on a smooth curve with a base point. Consider the open substack $$\mathcal{A}^{o}_{g,k}:=q^*(\mm_{0,2g+2k-2}\setminus \bigcup_{j \geq 3}B_j ) \setminus T_{base}.$$ The image of $\mathcal{A}^{o}_{g,k}$ under $\pi_k$ lies in the locus $\mm^{irr}_g$ of irreducible curves.\smallskip

 We have three boundary divisors $E_0, E_2, E_3$ on $\mathcal{A}^{o}_{g,k}$. Firstly, $E_0$ denotes the pullback of the boundary $\delta$ of $\mm^{irr}_g$. The general point of $E_3$ is the admissible cover corresponding to a finite cover $C \to \PP^1$ from a smooth curve $C$ and with a ramification profile $(3,1^{2g+2k-3})$ over some branch point, and simple branching over all other branch points. The general point of $E_2$ corresponds to a finite cover $C \to \PP^1$ with ramification profile $(2,2,1^{2g+2k-4})$. Denote by $$\mathcal{B}_{g,k}:=\mathcal{A}_{g,k}/\mathfrak{S}_{2g+2k-2},$$ the space of admissible covers with unordered branching. We set  $\mathcal{B}^{o}_{g,k}:=\mathcal{A}^{o}_{g,k}/\mathfrak{S}_{2g+2k-2}$ and let $D_0,D_2$ resp.\  $D_3$ denote the reduced images of $E_0, E_2$ resp.\ $E_3$ in $\mathcal{B}^o_{g,k}$. We write $\lambda$ for the Hodge class on both $\mathcal{A}^{o}_{g,k}$ and $\mathcal{B}^{o}_{g,k}$. Recall the following computation \cite[Prop.\ 11.1]{farkas-rimanyi}:
\begin{prop}[Farkas--Rim\'anyi] \label{can-bundle-formula}
We have the following canonical bundle formula
\begin{align*}K_{ \mathcal{B}^{o}_{g,k}}&=\frac{1}{2}[-\frac{2g+2k-1}{2g+2k-3}D_0-\frac{4}{2g+2k-3}D_2+\frac{2g+2k-9}{2g+2k-3}D_3]\\
&=8\lambda+\frac{D_3}{6}-\frac{3D_0}{2}
\end{align*}
\end{prop}\smallskip

We make a remark about the comparison between $\mathcal{B}^{o}_{2k-1,k}$ and $\H(1)$. Let $\mathcal{B}'\seq\mathcal{B}^{o}_{2k-1,k}$ be the open locus of admissible covers $f: C \to T$ such that the stabilization $\widehat{C}$ of $C$ is irreducible. Consider the open subset $\mm^{irr}_{2k-1,3}$ of irreducible marked curves and let $\mathcal{B}''\seq \mathcal{B}^{o}_{2k-1,k}\times_{\mm_{2k-1}} \mm^{irr}_{2k-1,3}$ denote the locus where the markings $p,q,r \in \widehat{C}$ avoid the image of unstable components and where $f(p),f(q),f(r)$ are distinct points in the image $T$ with $f$ unramified near $p,q,r$. There is a rational map $$\mathcal{B}'' \dashrightarrow \H(1)$$ extending to an isomorphism outside a codimension two set, cf.\ \cite[\S 3.5]{patel-thesis}.

\section{Cycle Computations} \label{comp}
\subsection{The Brill--Noether cycles} Starting with $\pi_k: \H(1) \to \H(0)$ we inductively define schemes $\H(m)$ and projective immersions $p_i^{(m)}: \H(m) \to \H(m-1)$, $m \geq 1$ for $i=1,2$ following the procedure of the previous section. We set $p_1^{(0)}=p_2^{(0)}=\pi_k$. Define $\mathbf{Z}(m+1)$ and $\widetilde{p}_i^{(m+1)}: \mathbf{Z}(m+1) \to \H(m)$ via the fibre product diagram
\[ \begin{tikzcd}
\mathbf{Z}(m+1) \arrow{r}{\widetilde{p}_2^{(m+1)}}  \arrow[swap]{d}{\widetilde{p}_1^{(m+1)}} & \H(m)\arrow{d}{p_1^{(m)}} \\%
\H(m)  \arrow{r}{p_1^{(m)}}& \H(m-1).
\end{tikzcd}
\]
We set $\H(m+1) := \mathbf{Z}(m+1) \setminus \Delta_{p_1^{(m)}},$ and define $p_i^{(m+1)}$ as the restriction of $\widetilde{p}_i^{(m+1)}: \mathbf{Z}(m+1) \to \H(m)$ to $\H(m+1)$, for $i=1,2$.

Note we have followed Herbert's procedure in constructing $\H(m)$ as multiple point loci. The image of the Brill--Noether cycle $\H(m)$ in $\H(0)$ coincides set theoretically with the locus of points over which the fibre of $\pi_k$ contain at least $m$ points. 
\begin{definition} \label{BN-cyc-defin}
Define the Brill--Noether cycles as
$$\mathcal{BN}_{m}=p^{(2)}_{1*} \ldots p^{(m)}_{1*}[\H(m)] \in A^m(\H(1)).$$
\end{definition}
Note that the diagonal $\Delta_{p_1^{(m)}}$ is both open and closed within $ \mathbf{Z}(m+1) $. Although the morphism $\widetilde{p}_1^{(m+1)}$ is surjective, after removing the diagonal, the image of $p_1^{(m+1)}$  is of codimension one. As a result, the dimension of the cycles $\mathcal{BN}_{m}$ drops by one at each step of the construction, so that we end up in the correct Chow group. See the also the discussion at the beginning of Section \ref{prelims}.

\subsection{The Eagon--Northcott cycles} Recall the kernel bundle description of Koszul cohomology, \cite{aprodu-nagel}. Let $X$ be a projective variety, $L \in \text{Pic}(X)$ be globally generated. Define $M_L$ via the exact sequence
$$0 \to M_L \to H^0(X,L) \otimes \mathcal{O}_X \xrightarrow{ev} L \to 0.$$
Then 
\begin{align*}
K_{p,q}(X,L) &\simeq \text{Coker} (\bigwedge^{p+1}H^0(L) \otimes H^0(L^{q-1}) \to H^0(\bigwedge^p M_L \otimes L^q) )\\
&\simeq \text{Ker} (H^1(\bigwedge^{p+1} M_L \otimes L^{q-1}) \to \bigwedge^{p+1} H^0(L) \otimes H^1(L^{q-1}))
\end{align*}

The universal stable map gives a universal cover $$\xymatrix{
\mathcal{C} \ar[r]^-f \ar[rd]_{\nu} &\mathcal{P} \ar[d]^{\mu}\\
& \widetilde{\H}(1),
}$$
where $\mathcal{P}:=\PP^1_{\widetilde{\H}(1)}$.
We have
$$ 0 \longrightarrow \mathcal{E}_{f} \longrightarrow f_{*} \omega_{f} \longrightarrow \mathcal{O}_{\mathcal{P}} \longrightarrow 0,$$
where $\mathcal{E}_{f}$ is the universal Tschirnhausen bundle. We further have the projective bundle
$\varphi: \; \mathcal{X}:=\PP(\mathcal{E}_f \otimes \omega_{\mu}) \to \mathcal{P}$
and a closed immersion
$\iota : \mathcal{C} \hookrightarrow \mathcal{X}.$ Set $h:= \mu\circ \varphi: \;  \mathcal{X} \to \widetilde{\H}(1).$
Define the universal kernel bundle $\mathcal{M}_{\mathcal{X}} $ by
\begin{align} \label{kernel-def}
0 \longrightarrow \mathcal{M}_{\mathcal{X}} \longrightarrow h^*h_*(\OO_{\mathcal{X}}(1)) \rightarrow \mathcal{O}_{\mathcal{X}}(1)\longrightarrow 0.
\end{align}
For all integers $i,j$, define sheaves $A^{[i,j]}[m], B^{[i,j]}[m]$ inductively on $\widetilde{\H}(m)$. Set $$A^{[i,j]}[1]:=h_*(\bigwedge^i \mathcal{M}_{\mathcal{X}}(j)), \;  \; \; B^{[i,j]}[1]:=\bigwedge^i h_*(\mathcal{O}_{\mathcal{X}}(1)) \otimes h_*( \mathcal{O}_{\mathcal{X}}(j)).$$
Define
\begin{align*}A^{[i,j]}[m]&:=\psi^*_1(m)A^{[i,j]}[m-1] \oplus \psi^*_2(m)\cdots \psi^*_2(2)A^{[i,j]}[1], \\ B^{[i,j]}[m]&:=\psi^*_1(m)B^{[i,j]}[m-1] \oplus \psi^*_2(m)\cdots \psi^*_2(2)B^{[i,j]}[1],
\end{align*}
where $\psi_{i}(m): \widetilde{\H}(m) \to \widetilde{\H}(m-1)$, $i=1,2$ are defined in Section \ref{notation}.

Let $\nu_m: \mathcal{C}_m \to \widetilde{\H}(m)$ be the universal curve, which is given by the fibre product
\[ \begin{tikzcd}
\mathcal{C}_m \arrow{r}{\mu^{(m)}_i}  \arrow[swap]{d}{\nu_m} & \mathcal{C}_{m-1} \arrow{d}{\nu_{m-1}} \\%
\widetilde{\H}(m) \arrow{r}{\psi_i(m)}& \widetilde{\H}(m-1),
\end{tikzcd}
\]
where $i$ can be either $1$ or $2$ in the horizontal arrows. Define kernel bundles $\mathcal{K}_m$ by
\begin{align*}
0 \longrightarrow \mathcal{K}_m \longrightarrow \nu^*_m \nu_{m*}(\omega_{\nu_m}) \rightarrow \omega_{\nu_m} \to 0.
\end{align*}
Notice that $\mathcal{K}_1 \simeq \iota^* \mathcal{M}_{\mathcal{X}}$, whereas $\mu_i^{(m)*} \mathcal{K}_{m-1}\simeq \mathcal{K}_m$.
Define sheaves $C^{[i,j]}[m], D^{[i,j]}[m]$
\begin{align*}
C^{[i,j]}[m]:=\nu_{m*}(\bigwedge^i \mathcal{K}_m \otimes \omega_{\nu_m}^{\otimes j}), \; \; \; 
D^{[i,j]}[m]:=\bigwedge^i \nu_{m*}(\omega_{\nu_m}) \otimes \nu_{m*}(\omega_{\nu_m}^{\otimes j})
\end{align*}
Define restriction maps
$$\beta^{[i,j]}[m]: A^{[i,j]}[m] \to C^{[i,j]}[m],$$
inductively. Set $\beta^{[i,j]}[1]$ to be the composition
 $$h_*\bigwedge^i \mathcal{M}_{\mathcal{X}} (j) \to h_* \iota_* \iota^* \bigwedge^i\mathcal{M}_{\mathcal{X}} (j) \simeq \nu_* \iota^* \bigwedge^i \mathcal{M}_{\mathcal{X}} (j)\simeq C^{[i,j]}[1].$$ For $m>1$, $\psi^*_1(m)\beta^{[i,j]}[m-1]$ gives a morphism
 $\psi^*_1(m)A^{[i,j]}[m-1] \to \psi^*_1(m)C^{[i,j]}[m-1]$. Composing this with the base change morphism yields a morphism $\psi^*_1(m)A^{[i,j]}[m-1] \to C^{[i,j]}[m]$. Secondly, composing $\psi^*_2(m)\cdots \psi^*_2(2)\beta^{[i,j]}[1]$ with the natural base change maps yields $\psi^*_2(m)\cdots \psi^*_2(2)A^{[i,j]}[1] \to C^{[i,j]}[m]$. Define $\beta^{[i,j]}[m]$ as the sum of these two maps. When there is no confusion, write $\beta^{[i,j]}$ for $\beta^{[i,j]}[m]$.
Similarly, there are maps 
$$\gamma^{[i,j]}[m]: B^{[i,j]}[m] \to D^{[i,j]}[m]$$
given as a sum of restriction maps, and we write $\gamma^{[i,j]}$ for $\gamma^{[i,j]}[m]$ where there does not seem to be any chance of confusion.

The short exact sequence (\ref{kernel-def}) induces
$$0 \to \bigwedge^i \mathcal{M}_{\mathcal{X}} (j) \to (h^*\bigwedge^i h_* \mathcal{O}_{\mathcal{X}}(1))\otimes \mathcal{O}_{\mathcal{X}}(j) \to \bigwedge^{i-1}\mathcal{M}_{\mathcal{X}} (j+1) \to 0.$$
If $j \geq 1$, the fact that the scroll $\mathcal{X}$ has a $2$-linear minimal free resolution given by the Eagon--Northcott complex implies $R^1 h_*\bigwedge^i \mathcal{M}_{\mathcal{X}} (j) =0$, \cite[\S 4]{lin-syz}. Hence we have exact sequences
$$ 0 \to A^{[i,j]}[1] \to B^{[i,j]}[1] \to A^{[i-1,j+1]}[1] \to 0,$$
provided $j \geq 1$. Pulling this back under the relevant projections and summing up we obtain
$$ 0 \to A^{[i,j]}[m] \to B^{[i,j]}[m] \to A^{[i-1,j+1]}[m] \to 0,$$
for $j \geq 1$.
We have the commutative diagram:
\begin{align} \label{comm1}
\xymatrix{
0 \ar[r]  & A^{[k-1,1]}[m]  \ar[r] \ar[d]^{\beta^{[k-1,1]}} & B^{[k-1,1]}[m]
\ar[r]  \ar[d]^{\gamma^{[k-1,1]}} &A^{[k-2,2]}[m] \ar[r] \ar[d]^{\beta^{[k-2,2]}} &0 \\
0 \ar[r] &C^{[k-1,1]}[m]   \ar[r] &D^{[k-1,1]}[m] \ar[r] &C^{[k-2,2]}[m] .
}.\end{align}

Define vector bundles $\widetilde{A}[m]$ and $\widetilde{C}[m]$ on $\widetilde{\H}(m)$ by
\begin{align*}
\widetilde{A}[1]&:= R^1 h_* \bigwedge^{k} \mathcal{M}_{\mathcal{X}}, \; \; \; \; \widetilde{A}[m]:=\psi^*_1(m)\widetilde{A}[m-1] \oplus \psi_2^*(m)\cdots \psi_2^*(2)\widetilde{A}[1] \\
\widetilde{C}[m]&:= R^1 \nu_{m*} \bigwedge^{k} \mathcal{K}_m.
\end{align*}
There is a natural morphism $\widetilde{\beta}: \widetilde{A}[m] \to \widetilde{C}[m]$.
We have the commutative diagram:
\begin{align} \label{comm2}
\xymatrix{
0 \ar[r]  & B^{[k,0]}[m]  \ar[r] \ar[d]^{\gamma^{[k,0]}} & A^{[k-1,1]}[m]
\ar[r]  \ar[d]^{\beta^{[k-1,1]}} &\widetilde{A}[m] \ar[r] \ar[d]^{\widetilde{\beta}} &0 \\
0 \ar[r] &D^{[k,0]}[m]   \ar[r] &C^{[k-1,1]}(m) \ar[r] &\widetilde{C}[m] .
}.\end{align}
We now define two classes in the $K$-group of $\widetilde{\H}(m)$. Let $\mathcal{V}[m]$ be the cokernel of the composition $\phi^{[k-1,1]}$ of the natural maps in the diagram below:
$$
\xymatrix@C+3.5pc{
  \text{Ker}\, \gamma^{[k-1,1]}/ \text{Ker} \, \gamma^{[k,0]} \ar[r]^-{\phi^{[k-1,1]}} \ar@{->>}[d] &A^{[k-2,2]}[m] \\
\text{Ker}\,\gamma^{[k-1,1]}/ \text{Ker}\,\beta^{[k-1,1]} \ar@{^{(}->}[r]& \text{Ker}\,\beta^{[k-2,2]}. \ar@{^{(}->}[u]
}
$$
Consider the morphism 
$$ \mathcal{F}[m]: \mathcal{V}[m] \to C^{[k-2,2]}[m]$$
induced by $\beta^{[k-2,2]}$. By relative duality 
$$\widetilde{C}[m]^* \simeq \nu_{m*} (\bigwedge^k \mathcal{K}^*_m \otimes \omega_{\nu_m}) \simeq C^{[k-2,2]}[m]\otimes \lambda^*,$$
for $\lambda:=c_1(\nu_{m*}(\omega_{\nu_m}))$. 
\begin{lem} \label{compzero}
The composition $$\widetilde{\beta}^* \otimes \lambda \circ \beta^{[k-2,2]}: A^{[k-2,2]}[m] \to \widetilde{A}[m]^* \otimes \lambda$$
of vector bundles is zero.
\end{lem}
\begin{proof}
By Grauert's theorem, $A^{[k-2,2]}[m]$ is locally free, see \cite[\S 4]{lin-syz}. At a closed point $p=[(f_i: C \to \PP^1)] \in \widetilde{\H}(m)$,
$$\text{Im}\, \beta_p^{[k-2,2]} \seq \text{Im}(\bigwedge^{k-1}H^0(\omega_C) \otimes H^0(\omega_C)) \seq H^0(\bigwedge^{k-2} M_{\omega_C} \otimes \omega_C^{\otimes 2})$$ by diagram (\ref{comm1}),
where $\beta_p^{[k-2,2]}:=\beta^{[k-2,2]} \otimes k(p)$. So it suffices to show
$\text{Im}(\bigwedge^{k-1}H^0(\omega_C) \otimes H^0(\omega_C)) \seq \text{Ker}\, \widetilde{\beta}^*_p$. 

We now argue as in the proof of Koszul duality, \cite[Thm.\ 2.2.4]{aprodu-nagel} using kernel bundles. From diagram (\ref{comm2}), $\text{Im}\, \widetilde{\beta}_p \seq K_{k-1,1}(C, \omega_C) $. But the vector space $K_{k-1,1}(C, \omega_C)$ is isomorphic to
\begin{align*}
\text{Ker}\left(H^1(\bigwedge^k M_{\omega_C}) \to ( \bigwedge^{k-1}H^0(\omega_C) \otimes H^0(\omega_C) )^{\vee}\right),
 \end{align*}
 using the natural isomorphism $\left( \bigwedge^{k-1}H^0(\omega_C) \otimes H^0(\omega_C) \right)^{\vee} \simeq \bigwedge^k H^0(\omega_C) \otimes H^1(\mathcal{O}_C)$. This completes the proof.
\end{proof}
We define:
\begin{align*}
\mathcal{W}[m]&:= \text{Ker} \, \widetilde{\beta}^* \otimes \lambda \; \; \text{where $\lambda:=c_1(\nu_{m*} \omega_{\nu_m})$}.
\end{align*}
By the above lemma, we have a morphism of sheaves
$$ \mathcal{F}[m]: \mathcal{V}[m] \to \mathcal{W}[m].$$
A point $p \in \widetilde{\H}(m)$ defines an $m$-tuple $[(g_i: C \to \PP^1)]$. Over $p$ we, have a commutative diagram
$$
\xymatrix{
0 \ar[r]  & \bigoplus^m_{i=1} \bigwedge^k H^0(X_i, \mathcal{O}(1)) \ar[r] \ar@{->>}[d]^{\gamma_p^{[k,0]}}& \bigoplus^m_{i=1}  H^0(X_i, \bigwedge^{k-1}\mathcal{M}_{X_i}(1))  
\ar[r]  \ar[d]^{f_1} & \bigoplus^m_{i=1} K_{k-1,1}(X_i,\mathcal{O}(1)) \ar[r] \ar[d]^{f_2} &0 \\
0 \ar[r]  & \bigwedge^k H^0(C, \omega_C) \ar[r] & H^0(C, \bigwedge^{k-1}\mathcal{M}_{\omega_C}(\omega_C))  
\ar[r]  &  K_{k-1,1}(C,\omega_C) \ar[r] &0 .
},$$
where $X_1, \ldots, X_m$ are the scrolls associated to the $m$ minimal pencils of $C$, and each component of the vertical maps are induced from the closed immersions $C \hookrightarrow X_i$ and $\gamma_p^{[k,0]}:=\gamma^{[k,0]}\otimes k(p)$. 
\begin{prop} \label{loc-free-crit}
Assume the map $f_2$ as above is injective at $p$. Then $\mathcal{V}[m]$ and $\mathcal{W}[m]$ are vector bundles of the same rank near $p$, and $\mathcal{F}[m]: \mathcal{V}[m] \to \mathcal{W}[m]$ is an isomorphism at $p$ if and only if $b_{k-1,1}(C,\omega_C)=m(k-1)$.
\end{prop}
\begin{proof}
By Grauert's theorem, $B^{[k-1,1]}[m]$, $D^{[k-1,1]}[m]$ are vector bundles and $\text{Ker}\, \gamma^{[k-1,1]}$ is locally free of rank $(m-1)(2k-1) {2k-1 \choose k-1}$. Likewise, $\text{Ker}\, \gamma^{[k,0]}$ is locally free of rank $(m-1){2k-1 \choose k}$, and $\text{Ker}\, \gamma^{[k-1,1]}/ \text{Ker}\, \gamma^{[k,0]}$ is locally free of rank $(m-1)(2k-2) {2k-1 \choose k}$. The morphism $\phi^{[k-1,1]}$ is a map between vector bundles.

As $\gamma_p^{[k,0]}$ is surjective, surjectivity of $f_1$ is equivalent to that of $f_2$ by the snake lemma. Further, as we are assuming $f_2$ is injective, $\text{Ker}\,\gamma_p^{[k,0]}=\text{Ker} \,f_1=\text{Ker} \beta^{[k-1,1]}_p$. In particular, $\phi_p^{[k-1,1]}:=\phi^{[k-1,1]} \otimes k(p)$ is injective and $\mathcal{V}[m]:=\text{Coker}\, \phi^{[k-1,1]}$ is locally free of rank
$$ (2k-2){2k-1 \choose k}-m(k-1) $$
near $p$ (use \cite[Lemma 4.6]{lin-syz}). The morphism $\widetilde{\beta}^* \otimes \lambda: \widetilde{C}[m]^* \otimes \lambda \to \widetilde{A}^*[m] \otimes \lambda$
is a map of vector bundles, and from 
$$
\xymatrix{
0 \ar[r]  & \bigoplus^m_{i=1} \bigwedge^k H^0(X_i, \mathcal{O}(1)) \ar[r] \ar@{->>}[d]^{\gamma_p^{[k,0]}}& \bigoplus^m_{i=1}  H^0(X_i, \bigwedge^{k-1}\mathcal{M}_{X_i}(1))  
\ar[r]  \ar[d]^{f_1} & \bigoplus^m_{i=1} H^1(X_i, \bigwedge^k M_{X_i}) \ar[r] \ar[d]^{\widetilde{\beta}_p} &0 \\
0 \ar[r]  & \bigwedge^k H^0(C, \omega_C) \ar[r] & H^0(C, \bigwedge^{k-1}\mathcal{M}_{\omega_C}(\omega_C))  
\ar[r]  &  H^1(C, \bigwedge^k M_{\omega_C}),
},$$
 injectivity of $f_2$ at $p$ implies $\widetilde{\beta}_p$ is injective. Thus $(\widetilde{\beta}^* \otimes \lambda)_p$ is surjective and $\mathcal{W}[m]$ is locally free of rank
$$(4k-2){2k-2 \choose k}-m(k-1)=(2k-2){2k-1 \choose k}-m(k-1) $$
near $p$. We have a commutative diagram
$$
\xymatrix{
 \bigoplus^m_{i=1}  H^0( \bigwedge^{k-1}\mathcal{M}_{X_i}(1)) \ar@{^{(}->}[r] \ar[d]^{f_1}& \bigoplus^m_{i=1}  \bigwedge^{k-1}H^0(\mathcal{O_{X_i}}(1)) \otimes H^0(\mathcal{O}_{X_i}(1))
\ar@{->>}[r]  \ar@{->>}[d]^{\gamma_p^{[k-1,1]}} &\bigoplus^m_{i=1}  H^0(\bigwedge^{k-2}\mathcal{M}_{X_i}(2)) \ar[d]^{\beta_p^{[k-2,2]}} \\
  H^0( \bigwedge^{k-1}\mathcal{M}_{\omega_C}(\omega_C)) \ar@{^{(}->}[r] &  \bigwedge^{k-1}H^0(\omega_C) \otimes H^0(\omega_C)
\ar[r]  & H^0(\bigwedge^{k-2}\mathcal{M}_{\omega_C}(\omega_C^{\otimes 2})) .
},$$
with exact rows, cf.\ \cite[Lemma 4.4]{lin-syz}. As $\gamma_p^{[k-1,1]}$ is surjective, surjectivity of $f_1$ is equivalent to surjectivity of the (injective) map
$$ \text{Ker}\, \gamma_p^{[k-1,1]}/\text{Ker}\, \gamma_p^{[k,0]} \to \text{Ker}\, \beta_p^{[k-2,2]},$$
or equivalently that the composition $$\mathcal{V}_p[m] \to A_p^{[k-2,2]}(m) / \text{Ker} \,\mathcal{\beta}_p^{[k-2,2]} \to H^0(C,\bigwedge^{k-2}\mathcal{M}_{\omega_C}(\omega_C^{\otimes 2}))$$ is injective. By Lemma \ref{compzero}, the image of $\mathcal{V}_p[m]$ lies in $\mathcal{W}_p[m] \seq H^0(C,\bigwedge^{k-2}\mathcal{M}_{\omega_C}(\omega_C^{\otimes 2}))$, and this completes the proof.
\end{proof}
Make the following definitions
\begin{align*}
\mathbf{V}[m]&:=c_1(A^{[k-2,2]}[m]-\text{Ker}\, \gamma^{[k-1,1]}+\text{Ker}\,\gamma^{[k,0]} ),\\
\mathbf{W}[m]&:=c_1(C^{[k-2,2]}[m]+\widetilde{A}[m] \cdot \lambda^*)
\end{align*}
As in the proof of the proposition above, if the morphism $f_2$ is injective at a point $p$ then $\mathbf{V}[m]$ resp.\  $\mathbf{W}[m]$ agrees with $c_1(\mathcal{V}[m])$ resp.\ $c_1(\mathcal{W}[m])$ about $p$. The proposition above justifies the following definition:
\begin{definition}
We define the Eagon-Northcott cycles
\begin{align*}
 \widetilde{\mathcal{EN}}_{m}&:=\mathbf{W}[m]- \mathbf{V}[m] \in A^{1}(\widetilde{\H}(m)), \\
 \mathcal{EN}_{m}&:=\psi_{1*}(2) \ldots \psi_{1*}(m)(\widetilde{\mathcal{EN}}_{m})_{|_{\H(1)}} \in A^{m+1}(\H(1)) 
 \end{align*}
\end{definition}
\subsection{Computations}
The following lemma is useful for induction arguments.
\begin{lem} \label{ind-lem}
The following formulae hold
\begin{align*}
\mathbf{V}[m+1]-\psi_1^*(m+1) \mathbf{V}[m]&=c_1(\psi_2^*(m+1) \ldots \psi_2^*(2)A^{[k-2,2]}[1]-\psi_{2}^*(m+1) \ldots \psi_{2}^*(2)B^{[k-1,1]}[1]\\
&+\psi_2^*(m+1) \ldots \psi_2^*(2)B^{[k,0]}[1])\\
\mathbf{W}[m+1]-\psi_1^*(m+1) \mathbf{W}[m]&= c_1(\psi_2^*(m+1) \ldots \psi_2^*(2) \widetilde{A}[1] \cdot \lambda^*)
\end{align*}
\end{lem}
\begin{proof}
These follow from the obvious identities \begin{align*}A^{[i,j]}[m+1]/ \psi_1^*(m+1)A^{[i,j]}[m]&=\psi_2^*(m+1) \ldots \psi_2^*(2)A^{[i,j]}[1]\\
B^{[i,j]}[m+1]/ \psi_1^*(m+1)B^{[i,j]}[m]&=\psi_2^*(m+1) \ldots \psi_2^*(2)B^{[i,j]}[1]\\
\widetilde{A}[m+1]/ \psi_1^*(m+1)\widetilde{A}[m]&=\psi_2^*(m+1) \ldots \psi_2^*(2)\widetilde{A}[1]
\end{align*}
 as well as the identities $\psi_1^*(m+1)D^{[i,j]}[m]=D^{[i,j]}[m+1]$, $\psi_1^*(m+1)C^{[k-2,2]}[m]=C^{[k-2,2]}[m+1]$, $\psi_1^*(m+1)\lambda=\lambda$ following from Grauert's theorem.
\end{proof}

Recall that $\H(1)$ is isomorphic in codimension two to $\mathcal{B}^{o}_{2k-1,k}\times_{\mm_{2k-1}}\mm_{2k-1,3}.$ We continue to write $D_0, D_2, D_3 \in A^1(\H(1), \mathbb{Q})$ for the pullback of the corresponding divisor classes from $A^1(\mathcal{B}^{o}_{2k-1,k}, \mathbb{Q})$ .
\begin{lem}
The following formulae hold in $A^1(\H(1), \mathbb{Q})$: 
\begin{align*} 
c_1(C^{[k-2,2]}[1])=&\frac{k+1}{(2k-3)(2k-1)}((2k-1)(4k-3){2k-2 \choose k-3}+(8k-3){2k-1 \choose k-2}) \lambda\\
&-\frac{k(k+1)}{2(2k-1)(2k-3)}{2k-1 \choose k-2} D_0, \\
c_1(A^{[k-2,2]}[1]+\widetilde{A}[1])=&2k{2k-2 \choose k-2} \lambda.
\end{align*}
\end{lem}
\begin{proof}
For $j \geq 0$, we have the short exact sequence:
$$ 0 \to C^{[k-2-j,2+j]}[1] \to D^{[k-2-j,2+j]}[1] \to C^{[k-3-j,2+j+1]}[1] \to 0.$$ 
Indeed, $R^1_{\nu_{1*}}(\bigwedge^{k-2-j}\mathcal{K}_1 \otimes \omega_{\nu_1}^{2+j})=0$, since $K_{k-3-j,3+j}(C, \omega_C)=0$ for any $[(C,q_1,q_2,q_3) \to \PP^1]  \in \H(1)$. Thus
$$ c_1(C^{[k-2,2]}[1])= \sum^{k-2}_{j=0}(-1)^j c_1(D^{[k-2-j,2+j]}[1]).$$
For $n \geq 2$, $\nu_{1*}(\omega^n_{\nu_1})=(6n^2-6n+1)\lambda-\frac{n^2-n}{2} D_0,$ see \cite[Ch.\ 13]{ACG2}. Thus
\begin{align*}
c_1(D^{[k-2-j,2+j]}[1])&=c_1(\bigwedge^{k-2-j} \nu_{1*} \omega_{\nu_1} \otimes \nu_{1*}(\omega^{2+j}_{\nu_1}))\\
=&{2k-1 \choose k-2-j}((6(2+j)^2-6(2+j)+1)\lambda-\frac{1}{2}((2+j)^2-(2+j))D_0) \\
&+(2j+3)(2k-2){2k-2 \choose k-3-j}\lambda,
\end{align*}
the first formula follows, using any computer algebra package. The second formula is an immediate consequence of the short exact sequences:
\begin{align*}
0 \to A^{[k-1,1]}[1] \to &B^{[k-1,1]}[1] \to A^{[k-2,2]}[1] \to 0,\\
0 \to B^{[k,0]}[1] \to &A^{[k-1,1]}[1] \to \widetilde{A}[1] \to 0
\end{align*}
together with the fact $h_* \mathcal{O}_{\mathcal{X}}(1) \simeq \nu_{1*}\omega_{\nu_1}$.
\end{proof}
Putting these facts together yields:
\begin{lem} \label{key-int-comp}
The following identities hold in $A^1(\H(1), \mathbb{Q})$
\begin{enumerate}[label=(\roman*)]
\item $c_1(C^{[k-2,2]}[1]-A^{[k-2,2]}[1]-\widetilde{A}[1])=(k-1)\pi_k^*\mathfrak{hur}$,
\item $c_1(\widetilde{A}[1] \cdot(1+ \lambda^*))=-(k-1)c_1(N_{\pi_k})$,
where $N_{\pi_k}$ is the relative normal bundle of $\pi_k$.
\end{enumerate}
\end{lem}
\begin{proof}
The first claim follows from the previous lemma together with the computation of $\mathfrak{hur}$ in \cite{ha-mu}. Note that the only boundary component in $\mm_g$ with nontrivial pullback to $\H(1)$ is $\delta_0$.

Lemma \ref{can-bundle-formula} plus the canonical bundle formula for $\mm_g$, \cite{ha-mu}, gives:
$$ N_{\pi_k}=-5\lambda-\frac{1}{2}D_0+\frac{D_3}{6}.$$
We need to show
$$c_1(\widetilde{A}[1])=(k-1)(3\lambda+\frac{D_0}{4}-\frac{D_3}{12}), $$
as $\widetilde{A}[1]$ has rank $k-1$, \cite[\S 4]{lin-syz}. From the exact sequence
$$0 \to \bigwedge^{n+1} \mathcal{M}_{\mathcal{X}}(-1) \to h^*\bigwedge^{n+1}h_*\mathcal{O}_{\mathcal{X}}(1) \otimes \mathcal{O}_{\mathcal{X}}(-1) \to \bigwedge^{n} \mathcal{M}_{\mathcal{X}} \to 0,$$
we get 
$$c_1(R^1h_*\bigwedge^{k} \mathcal{M}_{\mathcal{X}})
=c_1(R^2h_*\bigwedge^{k+1} \mathcal{M}_{\mathcal{X}}(-1))= \dots=c_1(R^{k-1}h_*( \det(\mathcal{M}_{\mathcal{X}})(-k+1))) .$$
Hence $c_1(\widetilde{A}[1])=c_1(\lambda \cdot R^{k-1}h_*((1-k)H))=(k-1)\lambda+c_1(R^{k-1}h_*((1-k)H))$, where $H$ denotes the class of $\mathcal{O}(1)$ on the projective bundle $\varphi: \mathcal{X} \to \mathcal{P}$. From 
$$0 \to \mathcal{O}_{\mathcal{X}} \to \varphi^*(\mathcal{E}^*_f \otimes \omega^*_{\mu}) \otimes \mathcal{O}(1) \to T_{\varphi} \to 0, $$
we deduce $\omega_{\varphi}=(1-k)H \otimes \det(\varphi^*(\mathcal{E}_f \otimes \omega_{\mu})).$
By relative duality, $$R^{k-1}h_*((1-k)H) \simeq (h_*\varphi^*(\det(\varphi^*(\mathcal{E}_f \otimes \omega_{\mu}))\otimes \omega_{\mu}))^*=(\mu_*(k\omega_{\mu}\otimes \det \mathcal{E}_f ))^*.$$
By Grothendieck--Riemann--Roch,
\begin{align*}
c_1(\mu_*(k\omega_{\mu}\otimes \det \mathcal{E}_f ))=\mu_*[&(k-1+(kc_1(\omega_{\mu})+c_1(\mathcal{E}_f))+\frac{1}{2}(kc_1(\omega_{\mu})+c_1^2(\mathcal{E}_f))^2+\ldots)\\
&\cdot(1-\frac{c_1(\omega_{\mu})}{2}+\frac{c_1^2(\omega_{\mu})}{12}+\ldots)]_2 
\end{align*}
which equals $\frac{1}{2}\mu_* c_1^2(\mathcal{E}_f)+\frac{2k-1}{2}\mu_*(c_1(\omega_{\mu}) \cdot c_1(\mathcal{E}_f))$
as $c_1^2(\omega_{\mu})=0$ for $\mathcal{P} \simeq \PP^1 \times \H(1).$
From \cite[Prop.\ 4.1]{patel-thesis},
\begin{align*}
\mu_* c_1^2(\mathcal{E}_f)&= \frac{18b}{10} \lambda -\frac{2b}{b-10}D_0+\frac{b}{2(b-10}D_2 \\
\mu_*(c_1(\omega_{\mu}) \cdot c_1(\mathcal{E}_f)) &=\frac{1}{2} \mu_*(c_1(\omega_{\mu}) \cdot f_*c_1(\omega_f)) \\
&=\frac{1}{2} \nu_*(f^*c_1(\omega_{\mu}) \cdot c_1(\omega_f))\\
&=-\frac{2}{b}\mu_* c_1^2(\mathcal{E}_f)\\
&=\frac{-36}{b-10}\lambda+\frac{4}{b-10}D_0-\frac{1}{b-10}D_2,
\end{align*}
for $b=2g-2+2k$, where the second to last line is from the proof of \cite[Prop.\ 4.1]{patel-thesis} and where we used $c_1(\mathcal{E}_f)=c_1(f_*\omega_f)=\frac{1}{2}f_*(c_1(\omega_f))$. The claim now follows by combining the formulae.
\end{proof}

\noindent We now arrive at the main result of this section.
\begin{thm} \label{divisor-computation}
We have the equality of $1$-cycles
$$(k-1)\psi_{1*}(m+1)[\H(m+1)]= \mathbf{W}[m]- \mathbf{V}[m] \in A^1(\H(m)).$$
In particular, $\mathcal{EN}_{m}=(k-1)\mathcal{BN}_{m+1} \in A^{m+1}(\H(1)).$
\end{thm}
\begin{proof}
We prove the claim by induction. By the double point formula \cite[\S 9.3]{fulton}
$$\psi_{1*}(m+1)[\H(m+1)]=\psi_1^*(m)\psi_1^*(m)[\H(m)]-c_1(N_{\psi_1(m)}),$$
where $N_{\psi_1(m)}=\psi_1^*(m)T_{\H(m-1)}-T_{\H(m)}$ is the relative normal bundle of $\psi_1(m)$. When $m=1$,
$$\psi_{1*}(2)[\H(2)]=\pi_k^{*}\mathfrak{hur}-c_1(N_{\pi_k}).$$

To begin the induction, we need to show
$$\mathbf{W}[1]-\mathbf{V}[1]=(k-1)(\pi_k^* \mathfrak{hur}-c_1(N_{\pi_k})).$$
From the definitions, we have
\begin{align*}
\mathbf{W}[1]-\mathbf{V}[1]&=c_1(C^{[k-2,2]}(1)+\widetilde{A}(1) \cdot \lambda^*-A^{[k-2,2]}(1))
\end{align*}
as $\gamma^{[k-1,1]}$, $\gamma^{[k,0]}$ are isomorphisms for $m=1$. The claim follows from Lemma \ref{key-int-comp}.

Assume $(k-1)\psi_{1*}(m)[\H(m)]=\mathbf{W}[m-1]- \mathbf{V}[m-1].$ Then
\begin{align*}
(k-1)\psi_{1*}(m+1)[\H(m+1)]&=\psi^*_1(m)(\mathbf{W}[m-1]- \mathbf{V}[m-1])-(k-1)c_1(\omega_{\psi_1(m)}) \\
&=\mathbf{W}[m]-\mathbf{V}[m]-c_1((\psi^*_2(m) \ldots \psi^*_2(2) \widetilde{A}[1])\cdot \lambda^*-\psi^*_2(m) \ldots \psi^*_2(2)A^{[k-2,2]}[1]) \\
&-c_1(\psi^*_{2}(m) \ldots \psi^*_{2}(2)B^{[k-1,1]}[1]-\psi^*_2(m) \ldots \psi^*_2(2)B^{[k,0]}[1])\\
& -(k-1)c_1(N_{\psi_1(m)}),
\end{align*}
by Lemma \ref{ind-lem}. 
Observe that $N_{\psi_1(m)}=\psi^*_2(m)N_{\psi_1(m-1)}$ \cite[Prop.\ 4.7]{kleiman}, so
$$N_{\psi_1(m)}=\psi^*_{2}(m) \ldots \psi^*_{2}(2)N_{\pi_k}.$$ 
Hence it suffices to show
$$c_1(\widetilde{A}[1]\otimes \lambda^*-A^{[k-2,2]}[1]+B^{[k-1,1]}[1]-B^{[k,0]}[1])=-(k-1)N_{\pi_k}.$$ 
 Or, $c_1(\widetilde{A}[1]\cdot (\lambda^*+1))=-(k-1)N_{\pi_k},$  which follows from Lemma \ref{key-int-comp}.
\end{proof}

\section{Geometry of curves with multiple pencils}
\subsection{Pencils in Geometrically General Position} \label{GGP}
Let $C$ be a connected, nodal curve of genus $g$ with no non-separating nodes. Let $f_1, \ldots, f_m: C \to \PP^1$ be finite morphisms of degree $k$ and assume that if $L_i:=f^*_i \mathcal{O}_{\PP^1}(1)$, then $h^0(C,L_i)=2$ for $ 1 \leq i \leq m$. Denote by $X_{f_i}:=\PP(\mathcal{E}_{f_i}(-2))$ the scroll associated to $f_i$, and let $\widetilde{X}_{f_i} \seq \PP^{g-1}$ be the image of $X_{f_i}$ under
$$X_{f_i} \to \PP(H^0(\PP^1, \mathcal{E}_{f_i}(-2))),$$
see \cite[\S 4]{lin-syz}. Then $\widetilde{X}_{f_i}$ has rational singularities and $$K_{p,q}(\widetilde{X}_{f_i},(\mathcal{O}_{\PP^{g-1}}(1))_{|_{\widetilde{X}_{f_i}}}) \simeq K_{p,q}(X_{f_i},\mathcal{O}_{\PP(\mathcal{E}_{f_i}(-2))}(1)),$$
see \cite[\S 1]{schreyer1}. Then $\widetilde{X}_{f_i}$ is a (possibly singular) rational normal scroll of degree $g-k+1$ in $\PP^{g-1}$ for $1 \leq i \leq m$. If $\{ u,v \}$ is a basis for $H^0(C,L_i)$ and $\{ y_1, \ldots, y_{g+1-k} \}$ is a basis for $H^0(C,\omega_C \otimes L_i^{-1})$, then $\widetilde{X}_{f_i}$ is defined by the two-by-two minors of the matrix of linear
forms
 $$\left (\begin{matrix}
uy_1 &    \ldots & uy_{g+1-k}\\
vy_1 &    \ldots & vy_{g+1-k}
\end{matrix} \right).$$

Choose distinct points $z_1, \ldots, z_{g-1-k} \in C$ such that $h^0(L_i + \sum_{j=1}^{g-1-k} z_j)=2$ for $1 \leq i \leq m$, or, equivalently, $h^0(C,\omega_C \otimes L_i^{-1}(-\sum_{j=1}^{g-1-k} z_j))=2$ (this holds for a general choice of $g-1-k$ points). Let $p: 
\PP^{g-1} \dashrightarrow \PP^{k}$ be the projection away from the points $z_1, \ldots, z_{g-1-k}$.  Then $p$ induces a composition of $g-1-k$ inner projections on each scroll $\widetilde{X}_{f_i}$. We obtain quadric hypersurfaces 
$$Q_{f_i}:=p(\widetilde{X}_{f_i}) \in |\mathcal{O}_{\proj^k}(2)|.$$ If $\{ s,t \}$ is a basis for $H^0(C,\omega_C \otimes L_i^{-1}(-\sum_{j=1}^{g-1-k} z_j))$ and $\{ u,v \}$ is a basis for $H^0(C,L_i)$, then $Q_{f_i}$ is the rank $4$ quadric in $\PP(H^0(\omega_C (-\sum_j z_j))^*)$ defined by the determinant of $\left (\begin{matrix}
us &     ut\\
vs &    vt
\end{matrix} \right).$ 

We make the following definition:
\begin{defin}
The degree $k$ pencils $f_1, \ldots, f_m: C \to \PP^1$,  are in \textbf{geometrically general position} with respect to $z_1, \ldots, z_{g-1-k} \in C$ if 
$$\dim \langle Q_{f_1}, \ldots, Q_{f_m} \rangle=m-1 $$ 
for the span $ \langle Q_{f_1}, \ldots, Q_{f_m} \rangle \seq |\mathcal{O}_{\proj^k}(2)| \simeq \PP^{\frac{(k+2)(k+1)}{2}}$.
\end{defin}
If $C$ is smooth then the quadrics $Q_{f_i}$ above can be obtained from the double points of the theta divisor $\Theta \seq \text{Pic}^{g-1}(C)$, cf.\ \cite{green-quadrics}. To see this, choose $g-1-k$ distinct points $z_1, \ldots, z_{g-1-k}$ on $C$ and let $D=\sum^{g-1-k}_{i=1} z_i$. Since $h^1(C,L_i)=g+1-k$ then for a general choice of the points $z_j$, each pencil $L_i+D$ is a $g^1_{g-1}$ and hence defines a double point of $\Theta$. The Theorem of Andreotti--Mayer and Kempf, \cite{andreotti-mayer}, \cite{kempf} states that the projectivized tangent cones of the theta divisor at these double points are rank $4$ quadrics $\widetilde{Q}_{f_1}, \ldots, \widetilde{Q}_{f_m}$ containing $C$. From the determinantal descriptions above, $\widetilde{Q}_{f_i}$ is the cone over $Q_{f_i}$, for $1 \leq i \leq m$. 

From \cite[\S 2]{schreyer1},  
$$\widetilde{Q}_{f_i}=\bigcup_{D \in |L_i|} \langle D+\sum_{j=1}^{g-1-k}z_j \rangle.$$
Let $\widetilde{C}$ denote the connected, nodal curve of compact type obtained by attaching rational tails $E_1, \ldots, E_{g-1-k}$ to $C$, with $E_j \cap C=z_j$ for $1 \leq j \leq g-1-k$. We have a finite morphism $g_i: \widetilde{C} \to \PP^1$ of degree $g-1$, with $(g_i)_{|_C}=f_i$, $\deg (g_i)_{|_{E_j}}=1$, $1 \leq j \leq g-1-k$. Let $H \in \text{Pic}(\widetilde{C})$ denote the line bundle with
$$H_{E_j} \simeq \mathcal{O}_{E_j}, \; \;  H_{C} \simeq \omega_C,$$
and consider the vector bundle $g_{i*}H$. The Mayer--Vietoris sequence gives $h^0(\widetilde{C},\omega_{\widetilde{C}}(-H))=1$. One sees
$$H^0(\PP^1, (g_{i*}H)^{\vee}(-2)) \neq 0,$$
and hence we may write $g_{i*}H=V_i \oplus \mathcal{O}_{\PP^1}(-2).$ Similar considerations show that $V_i$ is globally generated, and we have a morphism $\PP(V_i) \to \PP^{g-1}$, cf.\ \cite[\S 4]{lin-syz}. Since
$$h^0(\widetilde{C},H)=h^0(C,\omega_C), \; \; \; h^0(\widetilde{C},H\otimes g_i^*\mathcal{O}_{\PP^1}(-n))=h^0(C,\omega_C\otimes L_i^{\otimes-n}(-\sum_{j=1}^{g-1-k}z_j )),$$
then, by comparing with \cite[\S 2]{schreyer1}, one sees that the image of $\PP(V_i) \to \PP^{g-1}$ defines a scroll of the same type as $\widetilde{Q}_{f_i}$. The image of $\PP(V_i) \to \PP^{g-1}$ is therefore the quadric $\widetilde{Q}_{f_i}$, and further we have a natural morphism $\widetilde{C} \to \PP(V_i)$ induced by $g_i^*V_i \twoheadrightarrow \mathcal{O}_{\widetilde{C}}$. The image of the composition $\widetilde{C} \to \PP(V_i) \to \PP^{g-1}$ is the canonically embedded curve $C$. \smallskip

Recall now Ehbauer's notion of projection of syzygies, \cite{ehbauer}, \cite{aprodu-higher}. Let $V$ be a vector space and $X \seq \PP(V^*)$ be a projective variety. A point $x \in X$, corresponds to an exact sequence
$$0 \to W_x \to V \to \C \to 0.$$ Let $p_x: \proj(V^*) \dashrightarrow \proj(W^*_x)$ denote the projection centered in $x$, let $Y \seq \PP(W^*_x)$ be the projection of $X$, and let $S_X$, $S_Y$ be the corresponding homogeneous coordinate rings. The sequence
$$0 \to \bigwedge^p W_x \to \bigwedge^p V \to \bigwedge^{p-1} W_x \to 0$$
induces a map $p_x: K_{p,1}(S_X, V) \to K_{p-1,1}(S_X, W_x)$ on Koszul cohomology. There is an injective morphism $S_Y \hookrightarrow S_X$ inducing an injective map $K_{p-1,1}(S_Y, W_x) \hookrightarrow K_{p-1,1}(S_Y, W_x) $. Ehbauer's Lemma states that the image of $p_x$ lies in $K_{p-1,1}(S_Y, W_x)$, so we have a map
$$p_x: K_{p,1}(S_X, V) \to K_{p-1,1}(S_Y, W_x).$$
If $X$ is connected, reduced, and $L$ is base-point free then there is a natural isomorphism
$$K_{p,1}(S_{X'}, V) \simeq K_{p,1}(X,L)$$
for $V=H^0(X,L)$ and $X':=\phi_L(X)$. 


We now relate the assumption that minimal pencils are in geometrically general position to syzygies of a nodal curve $C$. Recall that we have a natural restriction map
$$j: \bigoplus_{i=1}^m K_{g-k,1}(X_{f_i},\mathcal{O}_{X_{f_i}}(1)) \to K_{g-k,1}(C,\omega_C),$$
where $X_{f_i}$ are the scrolls associated to the minimal pencils $f_i$ as above.
\begin{prop} \label{injectivity-syzy-scrolls}
Let $C$ be nodal, connected, with non-separating nodes of gonality $k\geq 3$. Assume we have a set of minimal pencils $f_1, \cdots, f_m$ with $h^0(C,f_i^*\mathcal{O}_{\PP^1}(1))=2$ for all $1 \leq i \leq m$. Let $D$ be a general effective, reduced divisor of degree $g-1-k$ with support in the smooth locus of $C$ and with $h^0(f_i^*\mathcal{O}_{\PP^1}(1)(D))=2$ for all $i$. Assume $\{f_1, \ldots, f_m \}$ is geometrically in general position with respect to $D$. Then the linear map $j$ as above is injective.
\end{prop}
\begin{proof}
Let $D=z_1+\ldots+z_{g-1-k}$. The map $j$ is injective on each factor $K_{g-k,1}(X_{f_i},\mathcal{O}_{X_i}(1))$, cf.\ \cite[Lemma 4.4]{lin-syz}. Suppose $j(v_1)+\ldots+j(v_m)=0$ for $v_i \in K_{g-k,1}(X_{f_i},\mathcal{O}_{X_{f_i}}(1))$, with $v_i$ not all zero for $1 \leq i \leq m$. Composing the projection maps from each $z_i$ we have a projection map $p: \PP^{g-1} \dashrightarrow \PP^{k}$ as well as the projection map $p: K_{g-k,1}(C,\omega_C) \to K_{1,1}(C,\omega_C(-D))$ on Koszul cohomology.  Each $p(j(v_i))$ may be considered as the equation of a quadric in $\PP^k$, and we have
$$p(j(v_1))+\ldots+ p(j(v_m))=0.$$  Each projection $Q_{f_i}:=p(\widetilde{X}_{f_i})$ is further itself a quadric hypersurface in $\PP^k$ containing the image $\phi_{\omega_C(-D)}(C)$, hence its defining equation $[Q_{f_i}]$ is an element of
$$\text{Ker}(\text{Sym}^2(H^0(\omega_C(-D))) \to H^0(2\omega_C(-D)))=K_{1,1}(C,\omega_C(-D)).$$
Further, $K_{1,1}(Q_{f_i},\mathcal{O}_{Q_{f_i}}(1))$ is the space of quadrics containing $Q_{f_i}$, and hence is spanned by $[Q_{f_i}]$. We have a commutative diagram
$$
\xymatrix{
K_{g-k,1}(\widetilde{X}_{f_i},\mathcal{O}_{\widetilde{X}_{f_i}}(1)) \simeq K_{g-k,1}(X_{f_i},\mathcal{O}_{X_{f_i}}(1)) \ar[r] \ar[d]  & K_{g-k,1}(C, \omega_C) \ar[d] \\
K_{1,1}(Q_{f_i},\mathcal{O}_{Q_{f_i}}(1)) \ar[r] & K_{1,1}(C,\omega_C(-D)).
}$$
Hence $p(j(v_i))$ is a scalar multiple of $[Q_{f_i}]$. Thus the assumption that the pencils are geometrically in general position implies $p(j(v_i))=0$ for all $i$. As $D$ is general, this in turn implies $v_i= 0$ for $1 \leq i \leq m$, \cite[Prop.\ 2.14]{aprodu-nagel}, which is a contradiction.
\end{proof}

\subsection{The Key Construction} \label{sect=key-const}
We now generalize a construction from \cite[\S 3]{lin-syz}. Let $C$ be an integral curve of genus $g \geq 3$ and gonality $k \leq \frac{g+1}{2}$, and choose a nonnegative integer $n \leq g+1-2k$. Choose pairs of distinct points $(x_i,y_i)$ on $C$ for $1 \leq i \leq n$, and let $D$ be the semistable curve of genus $g+n$ obtained by adjoining smooth rational curves $R_i$ to $C$ at $x_i,y_i$. Mark $C$ at three general points $p,q,r$ (in particular, $\text{Aut}[C,p,q,r]=\{id\}$). Let $f: C \to \PP^1$ be a morphism of degree $k$ with $(f(p),f(q),f(r))=(0,1,\infty)$ and $f(x_i) \neq f(y_i)$ for all $1 \leq i \leq n$.
\begin{defin}[The Key Construction] \label{key-const} 
  Let $C$, $D$ be as above. Construct a stable map $[h : D \to \PP^1] \in \cM^{ns}_{g+n,k+n}(\PP^1,\{0,1,\infty\})$ by setting $h_{|_C}=f: C \to \PP^1$ and choosing $h_{|_{R_i}}: R_i \to \PP^1$ to be an isomorphism.
\end{defin}
The stable map $[h]$ as above is a smooth point of $\cM^{ns}_{g+n,k+n}(\PP^1,\{0,1,\infty\})$, cf.\ Section \ref{inf-gp}.

Following \cite{lin-syz}, a smooth curve $C$ of genus $g \geq 3$ and gonality $k \leq \frac{g+1}{2}$ satisfies \emph{bpf-linear growth} if
\begin{align*}
&\dim G^1_{k+l}(C) \leq l, \; \; \text{for $0 \leq l \leq g+1-2k$} \\
\text{and, further,} \; \; & \dim G^{1,\mathrm{bpf}}_{k+l}(C)<l, \; \,\text{for $0 < l \leq g+1-2k$} 
\end{align*}
where $G^{1,\mathrm{bpf}}_{d}(C) \seq G^{1}_{d}(C)$ is the open locus of base point free pencils of degree $d$ on $C$. 

Let $C$ be a smooth curve of genus $g$ and gonality $k \leq \frac{g+1}{2}$, set $n=g+1-2k$ and let $(D,p,q,r) \in \mm_{g+n,3}$ be the marked, stable curve of genus $g+n$ constructed above. Assume $C$ satisfies bpf-linear growth and that the underlying set of $W^1_k(C)$ consists of $m$ points $A_1, \ldots, A_m$. To each minimal pencil $A_i$ we have a stable map $f_j: C \to \PP^1$ for $1 \leq i \leq m$. As in definition \ref{key-const}, we extend $f_j$ to stable maps $h_j: D \to \PP^1$. Let $\hat{D}$ denote the stabilization of $D$.
\begin{prop} \label{prop-hi}
  Let $C,D$ be as above and set $n=g+1-2k$. Assume the smooth curve $C$ satisfies bpf-linear growth and that the underlying set of $W^1_k(C)$ consists of $m$ pencils $A_1, \ldots, A_m$, each of which has ordinary ramification. Assume the points $(x_i, y_i)$ are general for all $i$. Then the underlying set of the fibre of $$\pi_{k+n}: \cM_{g+n,k+n}(\PP^1,\{0,1,\infty\}) \to \mm_{2(k+n)-1,3}$$ over $(\hat{D},p,q,r)$ consists of $m$ points, namely the maps $\{ h_j, 1 \leq j \leq m\}$ as defined above.
\end{prop}
\begin{proof}
  Let $\mathcal{B}_{g+n,k+n} \to \mm_{g+n}$ denote the space of degree $k+n$ admissible covers with unordered branch points. Arguing as in \cite[Prop.\ 3.5]{lin-syz}, there are precisely $m$ admissible covers over $\hat{D}$, corresponding to the torsion free sheaves $\mu_*A_1, \ldots, \mu_*A_m$ under the construction of \cite[Thm.\ 5]{ha-mu}, where $\mu: C \to \hat{D}$ is the normalization morphism. There is a proper, birational morphism $U \times_{\mm_{2(k+n)-1}} \mathcal{B}_{g+n,k+n} \to \pi_{k+n}^{-1}(U) $, defined in an open subset $U$ about $[(\hat{D},p,q,r)]\in \mm_{2(k+1)-1,3}$, sending an admissible cover $F: B \to T$ to the stable map obtained by marking the target $T$ at $F(p),F(q),F(r)$, forgetting other markings, stabilizing the target to produce a morphism $B \to \PP^1$ sending $(p,q,r)$ to $(0,1,\infty)$, and then stabilizing the morphism $B \to \PP^1$. In particular, each stable map in $\pi_{k+n}^{-1}([(\hat{D},p,q,r)])$ is produced by this procedure, which converts the admissible cover corresponding to $\mu_*A_i$ into the stable map $h_i$. This completes the proof.

\end{proof}

Choose $0 \leq n \leq 2k-5$ and, for any integer $m$, consider
$$\mm_{2k-1-n,k-n}((\PP^1)^m,\{0,1,\infty\}; 2n) \times \mm_{0,1}((\PP^1)^m;2)^{\times n}.$$
Denote the markings on $[f] \in \mm_{2k-1-n,k-n}((\PP^1)^m,\{0,1,\infty\}; 2n)$ as $x_1,y_1,\ldots ,x_n,y_n$. We have  $$ev_{x_i},ev_{y_i} \; : \; \mm_{2k-1-n,k-n}((\PP^1)^m,\{0,1,\infty\}; 2n) \to (\PP^1)^m, \; \; 1 \leq i \leq n$$
by evaluating at these markings. Let
$$U_n(m) \seq \mm_{2k-1-n,k-n}((\PP^1)^m,\{0,1,\infty\}; 2n)$$
denote the open locus of points $\alpha$ such that 
$$pr_j(ev_{x_i}(\alpha)) \neq pr_j(ev_{y_i}(\alpha))\; \text{ for all $1\leq j \leq m$, $1 \leq i \leq n$},$$
for $pr_j: (\PP^1)^m \to \PP^1$ the projection to the $j^{th}$ factor.
Further denote by $$ev_{1,j}, ev_{2,j} \; : \; \mm_{0,1}((\PP^1)^m;2)^{\times n} \to (\PP^1)^m, \; \; 1 \leq j \leq n$$
the projection to the evaluation morphisms on the $j^{th}$ factor of $ \mm_{0,1}((\PP^1)^m;2)^{\times n}$. Let 
$$\mathcal{M}^{sm}_{0,1}((\PP^1)^m;2)\seq \mm_{0,1}((\PP^1)^m;2)$$
denote the open locus of morphisms $\PP^1 \to (\PP^1)^m$ with smooth base.
We denote by $$ V_n(m) \seq U_n(m) \times \mathcal{M}^{sm}_{0,1}((\PP^1)^m;2)^{\times n}$$
the closed subset of points $(\alpha,\beta)$ with $ev_{x_i}(\alpha)=ev_{1,i}(\beta)$, $ev_{y_i}(\alpha)=ev_{2,i}(\beta)$ for all $1 \leq i \leq n$.
We have a glueing morphism
$$q_n(m): \; V_n(m) \to \mm_{2k-1,k} ((\PP^1)^m,\{ 0,1,\infty \}),$$
defined by glueing maps together in the obvious way. 
\begin{lem} \label{dim-M}
Fix integers $n,m$. Let $u=[f: B \to (\PP^1)^m] \in \mm_{g,k}((\PP^1)^m,\{0,1,\infty \}; 2n)$ with $f$ finite. Each component of $\mm_{g,k}((\PP^1)^m,\{0,1,\infty \}; 2n)$ containing $u$ has dimension at least $$3g-m(g+2-2k)+2n.$$
\end{lem}
\begin{proof}
Each component containing $[f: B \to (\PP^1)^m]$ has dimension at least
 \[ \dim T^1(B/(\PP^1)^m)-\dim T^2(B/(\PP^1)^m)+3+2n-3m\] where $T^i(B/(\PP^1)^m)$, $i=1,2$, are vector spaces fitting into the exact sequence
 \[ 0 \to T^0(B) \to H^0(B,f^*T_{(\PP^1)^m}) \to T^1(B/(\PP^1)^m) \to T^1(B) \to H^1(B,f^*T_{(\PP^1)^m}) \to T^2(B/(\PP^1)^m) \to 0,\]
 for $T^i(B):=\text{Ext}^i(\Omega_B,\mathcal{O}_B)$, see e.g.\ \cite[\S 2]{kemeny-thesis}. The term $3+2n$ in the above formula comes from the base points and markings and the term $3m$ corresponds to imposing that the base points are sent to $\{(0)^m,(1)^m,(\infty)^m\}$. Further, $T^1(B/(\PP^1)^m)$ represents first order deformations of $f$ (forgetting base points and markings) whereas $T^2(B/(\PP^1)^m)$ contains the obstructions. Thus
  \[ \dim T^1(B/(\PP^1)^m)-\dim T^2(B/(\PP^1)^m)=\chi(f^*T_{(\PP^1)^m})+\dim T^1(B)-\dim T^0(B) \] To compute $\dim T^1(B)-\dim T^0(B)$, we proceed as in \cite[Prop.\ 2.2.6]{kemeny-thesis}. As $f$ is finite and hence stable, some power of $\omega_B \otimes f^*\mathcal{O}_{(\PP^1)^m}(1)$ is very ample and gives an embedding $j: B \hookrightarrow \PP^N$. Then $(j,f): B \hookrightarrow \PP^N \times (\PP^1)^m$ is a closed immersion and one computes \[\dim T^0(B)-\dim T^1(B)=\chi(f^*T_{(\PP^1)^m})+\chi(T_{\PP^N_{|_B}})-\chi(N_{B,\PP^N \times (\PP^1)^m}).\] Using the Euler sequence, Riemann--Roch on $B$, the sequence
 $$0 \to f^*T_{(\PP^1)^m} \to N_{B,\PP^N \times (\PP^1)^m} \to N_{B,\PP^N} \to 0 $$ and the formula $\deg N_{B,\PP^N}=(N+1)\deg B+2g-2$, one computes that 
 each component about $[f: B \to (\PP^1)^m]$ has dimension at least $3g-m(g+2-2k)+2n$. 
\end{proof}
We now estimate the dimension of $V_n(m)$.
\begin{lem} \label{dim-V}
Fix $g,k,n,m$. Let $u=[(f,g_1, \ldots, g_n)] \in V_n(m)$ where $f: B \to (\PP^1)^m$ is finite.
\begin{enumerate}
\item \label{dim-est-J} Every component $J$ of $V_n(m)$ containing $u$ has dimension at least $6k-2n-m-3$.
\item \label{dim-exact-J} Assume in addition that $\dim_{[f]} \mm_{2k-1-n,k-n}((\PP^1)^m,\{0,1,\infty\}; 2n)=6k-m(n+1)-n-3$. Then $\dim J =6k-2n-m-3$ at $u$.
\end{enumerate}
\end{lem}
\begin{proof}
\emph{Part (\ref{dim-est-J}):} Let $J \seq V_n(m)$ be a component containing $u$. The $i^{th}$ factor of $\mathcal{M}^{sm}_{0,1}((\PP^1)^m;2)^{\times n}$ is irreducible of dimension
$$2+h^0(g_i^*T_{(\PP^1)^m})-\dim \text{PGL}(2)=3m-1.$$ Each component of $\mm_{2k-1-n,k-n}((\PP^1)^m,\{0,1,\infty \}; 2n)$ containing $[f: B \to (\PP^1)^m]$ has dimension at least
$6k-m(n+1)-n-3$ from Lemma \ref{dim-M}. Each condition $ev_{x_i}(\alpha)=ev_{1,i}(\beta)$, $ev_{y_i}(\alpha)=ev_{2,i}(\beta)$ reduces the dimension by at most $m$. Thus $J$ has dimension at least
 $$(6k-m(n+1)-n-3)+n(3m-1)-2nm=6k-2n-m-3.$$

 \emph{Part (\ref{dim-exact-J}):} Consider $pr_1 : V_n(m) \to U_n(m) \seq \mm_{2k-1-n,k-n}((\PP^1)^m,\{0,1,\infty\}; 2n)$. All fibres of $pr_1$ are pure of dimension $n[(3m-1)-2m]=n(m-1)$. The claim follows.

\end{proof}

\subsection{Infinitesimal General Position} \label{inf-gp}
If $f: C \to X$ is a stable map from a nodal curve to a smooth projective variety, then first order deformations of $f$ are described by $\text{Ext}^1_C(\Omega_f^{\bullet},\mathcal{O}_C)$ and obstructions are given by $\text{Ext}^2_C(\Omega_f^{\bullet},\mathcal{O}_C)$, where $\Omega^{\bullet}_f$ is the complex $$f^*\Omega_X \xrightarrow{df} \Omega_C,$$
supported in degrees $-1,0$. If $f$ is unramified at the generic point of each component of $C$, $\mathbb{R}\mathcal{H}\text{om}_{\mathcal{O}_C}(\Omega_f^{\bullet},\mathcal{O}_C)$ is quasi-isomorphic to $N_f[-1]$, where $N_f$ is the \emph{normal sheaf}, \cite[\S 4]{bogomolov-hassett-tschinkel}.

Let $C$ be a nodal curve and $T=\{p,q,r\} \seq C$ a marking in the smooth locus of $C$. Let $F: C \to (\PP^1)^m \in \mm_{g,k}((\PP^1)^m, \{0,1,\infty \})$ with base points $T$.
For any $\sigma=\{ \sigma_1, \ldots, \sigma_j \} \seq \{1, \ldots, m \}$, we have a projection
$$pr_{\sigma} \; : \; (\PP^1)^m \to (\PP^1)^{|\sigma|}.$$
Let $F_{\sigma}:=pr_{\sigma}  \circ F.$ We say that the set $F$ is \emph{successively unobstructed} if $$\text{Ext}^2_C(\Omega_{F_{\sigma}}^{\bullet},\mathcal{O}_C(-T))=0$$ for all subsets $\sigma \seq \{1, \ldots, m \}$. This is equivalent to requiring $h^1(N_{F_{\sigma}}(-T))=0$ for all $\sigma$, provided each $f_i:=pr_i \circ F$ is generically unramified. In the next few lemmas we will explore the deformation theoretic meaning of this definition.  
\begin{lem} \label{euler-char-comp-lemma-marked}
  Let $B$ be an connected, nodal curve of genus $g$. Let $$F: B \to (\PP^1)^m \in \mm_{g,k}((\PP^1)^m, \{0,1,\infty \}),$$ with base points in $T$, which we assume avoid all ramification of $f_i$, $1 \leq i \leq m$. Assume $f_1, \ldots, f_m $ are finite. Let $\nu: \widetilde{B} \to B$ be the normalization, with components $\widetilde{B}_1, \ldots, \widetilde{B}_r$, and let $\delta$ be the number of nodes of $B$. Then $F$ is successively unobstructed if and only if 
  $$h^0(N_{F_{\sigma}}(-T))=3(\delta+1)-|\sigma|(g+2-2k)-\sum_{i=1}^r(3-3g(\widetilde{B}_i))$$ for all $\sigma \seq \{1, \ldots, m\}$. In particular, assuming either:
  \begin{enumerate}
  \item $B$ is integral, or
  \item $B=C \cup \bigcup_{i=1}^{r-1} R_i$ for $C$ integral nodal, $R_i \simeq \PP^1$, $C \cap R_{i}=\{ x_i,y_i \}$, $ 1 \leq i \leq r-1$.
  \end{enumerate}
  Then $F$ is successively unobstructed if and only if $$h^0(N_{F_{\sigma}}(-T))=3g-|\sigma|(g+2-2k).$$
\end{lem}
\begin{proof}
  We compute $\chi(N_{F_{\sigma}}(-T))$. Firstly, $\chi(N_{F_{\sigma}}(-T))=\chi(N_{F_{\sigma}})-3(|\sigma|-1),$ so it suffices to compute $\chi(N_{F_{\sigma}})$.
 There is a short exact sequence
 $$ 0 \to F_{\sigma}^*T_{(\PP^1)^j} /T_B \to N_{F_{\sigma}} \to \mathcal{E}\text{xt}^1_{\mathcal{O}_B}(\Omega^1_B,\mathcal{O}_B) \to 0,$$
 where $j=|\sigma|$, see \cite[Pg.\ 541]{bogomolov-hassett-tschinkel}. The sheaf $\mathcal{E}\text{xt}^1_{\mathcal{O}_B}(\Omega^1_B,\mathcal{O}_B)$ is a skyscraper sheaf supported at the nodes of $B$, whereas $$T_B \simeq \nu_* T_{\widetilde{B}}(-\sum_{i=1}^{\delta}r_i+q_i),$$ where $\nu: \widetilde{B} \to B$ is normalization and $r_i,q_i$, lie over the nodes of $B$, \cite[\S 11.3]{ACG2}. The proof now follows by Riemann--Roch.
 \end{proof}
The following is very similar to \cite[Prop.\ 1.4]{li-tian} and the proof is left to the reader.
 \begin{prop} \label{tgt-space-base}
   Let $C$ be a nodal curve, let $T \seq C$ be a marking and $f: \; (C,T) \to X$ a stable map to a smooth projective variety. Then the space of first order deformations $F: \; (\mathcal{C},\mathcal{T}) \to X$ of $f$ such that $F_{|_{\mathcal{T}}}$ is constant is given by $\text{Ext}^1_C(\Omega^{\bullet}_f,\mathcal{O}_C(-T))$.
 \end{prop}

Combining Lemma \ref{dim-M}, Lemma \ref{euler-char-comp-lemma-marked} and Proposition \ref{tgt-space-base} we obtain:
\begin{cor} \label{cor-tgt-space-base}
Let $B$ be a connected nodal curve of genus $g$ and assume either:
  \begin{enumerate}
  \item $B$ is integral, or
  \item $B=C \cup \bigcup_{i=1}^{r-1} R_i$ for $C$ integral nodal, $R_i \simeq \PP^1$, $C \cap R_{i}=\{ x_i,y_i \}$, $ 1 \leq i \leq r-1$.
  \end{enumerate}
 Let $F: B \to (\PP^1)^m \in \mm_{g,k}((\PP^1)^m, \{0,1,\infty \}),$ with finite components $f_1, \ldots, f_m $ and base points $T$ avoiding the ramification of $\{f_i \}$. Then $F$ is successively unobstructed with respect to $T$ if and only if, for any $\sigma \seq \{1, \ldots, m\} $, $\mm_{g,k}((\PP^1)^{|\sigma |},\{0,1,\infty\})$ is smooth at $[F_{\sigma}]$ of dimension $3g-|\sigma|(g+2-2k)$. 
\end{cor}

Recall from Section \ref{notation} the morphism
$$\pi_k : \cM_{g,k}(\PP^1, \{0,1,\infty \}) \to \mm_{g,3}.$$
Let $C$ be an irreducible nodal curve of genus $g$ admitting a degree $k$ morphism to $\PP^1$.
\begin{prop} \label{self-trans-lemma}
  Let $C$ be an irreducible nodal curve of genus $g$ admitting a degree $k$ morphism to $\PP^1$. Assume that for a general marking $T$ of degree three and any $f: \widetilde{C} \to \PP^1\in \pi_k^{-1}[(C,T)]$, the base $\widetilde{C}$ is irreducible. Assume $\pi_k^{-1}[(C,T)] =\{f_1, \ldots, f_m \}$ is zero-dimensional of cardinality $m \geq 2$ and that $h^0(f_i^*\mathcal{O}_{\PP^1}(2))=3$ for $1 \leq i \leq m$. Set $$[F:=(f_i) \; : \; C \to (\PP^1)^m].$$ Then $\pi_k: \cM_{g,k}(\PP^1, \{0,1,\infty \}) \to \mm_{g,3}$ is self-transverse in an open subset about $[f]$ if and only if $F$ is successively unobstructed (with respect to $T$). 
\end{prop}
\begin{proof}
  Let $\sigma=\{ \sigma_1, \ldots, \sigma_j \} \seq \{1, \ldots, m\}$ and for any $\ell$ set $\hat{\sigma}_{\ell}:=\sigma \setminus \{\sigma_{\ell}\}$. Denote by
  $$d\pi_{\sigma}: \text{Ext}^1_C(\Omega^{\bullet}_{F_{\sigma}},\mathcal{O}_C(-T)) \to \text{Ext}^1_C(\Omega_C,\mathcal{O}_C(-T))$$
  the map taking a first order deformation of $F_{\sigma}: (C,T) \to (\PP^1)^j$ preserving $F_{\sigma}(T)$ to a deformation of the marked curve $(C,T)$. Then $d\pi_{\sigma}$ is induced from the triangle
  $$\Omega_C \to \Omega^{\bullet}_{F_{\sigma}}\to F^*_{\sigma} \Omega_{(\PP^1)^j}[1].$$
  In particular, since $$\text{Hom}_C(F_{\sigma}^*\Omega_{(\PP^1)^j},\mathcal{O}_C(-T))\simeq \bigoplus_{i=1}^jH^0(f^*_{\sigma_i}\mathcal{O}_{\PP^1}(2)(-T))=0$$
  for $T$ general (as $h^0(f^*_{\sigma_i}\mathcal{O}_{\PP^1}(2))=3$), $d\pi_{\sigma}$ is injective. Set $V_{\sigma}:=d\pi_{\sigma}(\text{Ext}^1_C(\Omega^{\bullet}_{F_{\sigma}},\mathcal{O}_C(-T)))$.
  
  The isomorphism $F^*_{\hat{\sigma}_{\ell}}\Omega_{({\PP^1})^{j-1}}\oplus f^*_{\sigma_{\ell}}\Omega_{\PP^1}\to F^*_{\sigma}\Omega_{({\PP^1})^j}$
  induces a triangle
  $$\Omega_C \to \Omega^{\bullet}_{F_{\hat{\sigma}_{\ell}}}\oplus \Omega^{\bullet}_{f_{\sigma_{\ell}}} \to \Omega^{\bullet}_{F_{\sigma}}.$$
  Applying the Ext functor, we obtain 
  \begin{align*}
    V_{\sigma}=V_{\hat{\sigma}_{\ell}}\cap V_{\{\sigma_{\ell}\}}=V_{\{\sigma_1\}} \cap\ldots\cap V_{\{\sigma_j\}}.
  \end{align*}
  Hence, $\{V_{\{1 \}},\ldots,V_{\{m \}} \}$ is in general position in $\text{Ext}^1_C(\Omega_C,\mathcal{O}_C(-T))=T_{(C,T)}\mm_{g,3}$ if and only if, for any $\sigma \seq \{1, \ldots, m\}$, 
  $V_{\sigma}$ has codimension $|\sigma|(g+2-2k)$ in $T_{(C,T)}\mm_{g,3}$, which is equivalent to requiring that $h^1(N_{F_{\sigma}}(-T))=0$ by Lemma \ref{euler-char-comp-lemma-marked}.

  To conclude, we observe that if $h^1(N_{F_{\sigma}}(-T))=0$ for all subsets $\sigma$, then the same holds for an open subset about $(C,T)$. This is immediate from the fact that $\pi_k$ is unramified over $[(C,T)]$, as shown above, together with Corollary \ref{cor-tgt-space-base} (as the smooth locus is open).
\end{proof}
\begin{defin} \label{def-inf-unob}
 Let $C$ be an irreducible nodal curve of gonality $k$ and let $T \seq C$ be a marking of degree three in the smooth locus. We say the minimal pencils of $C$ are ``infinitesimally in general position" with respect to $T$ if:
 \begin{enumerate}
 \item For any $[f: \widetilde{C} \to \PP^1]\in \pi_k^{-1}[(C,T)]$, the base $\widetilde{C}$ is irreducible, $h^0(f^*\mathcal{O}_{\PP^1}(2))=3$ and $f$ is etale in an open set about $T$.
 \item $\pi_k^{-1}[(C,T)] =\{f_1, \ldots, f_m \}$ is zero-dimensional and $F=(f_i): C \to (\PP^1)^m$ is successively unobstructed with respect to $T$.
 \end{enumerate}

\end{defin}

We now return to the key construction (Section \ref{sect=key-const}). Let $C$ be an integral curve of genus $g$ and gonality $k\leq \frac{g+1}{2}$. Let $f_1, \ldots, f_m: C \to \PP^1$ be degree $k$ morphisms, with $(f_j(p),f_j(q),f_j(r))=(0,1,\infty)$ for fixed points $p,q,r \in C$. Set $T=p+q+r$ and assume $n \leq g+1-2k$. Let $$F=(f_j) \; : C \to (\PP^1)^m.$$
For each $1 \leq i \leq n$, choose distinct points $(x_i,y_i)$ of $C$ with $f_j(x_i) \neq f_j(y_i)$ for all $i,j$. Choose $$\beta_i: (\PP^1,(x_i,y_i)) \to (\PP^1)^m, \; \; 1 \leq i \leq n$$
such that $\beta_i(x_i)=(f_j(x_i))_{1 \leq j \leq m}, \beta_i(y_i)=(f_j(y_i))_{1 \leq j \leq m}$ and let 
$$H: D \to (\PP^1)^m$$
be the stable map of degree $k+n$ and genus $g+n$ obtained by glueing $F$ and $\beta_i$, \S \ref{sect=key-const}.

\begin{prop} \label{key-const-succ-unobst}
  With notation as above, assume $F$ is successively unobstructed with respect to $T$. Assume that $(x_i,y_i)$ lies outside the ramification locus for $f_j$ as well as $H_{\sigma}$ for each $\sigma \seq \{1, \ldots, m \}$ with $|\sigma| \geq 2$. Then $H$ is successively unobstructed with respect to $T$.
\end{prop}
\begin{proof}
  Let $\sigma \seq \{1, \ldots, m\}$ and consider $H_{\sigma}=(h_{\sigma_i})_{1 \leq i \leq |\sigma|}$. We need $H^1(D,\mathcal{N}_{H_{\sigma}}(-T))=0$. When $|\sigma|=1$, $\mathcal{N}_{H_{\sigma}}$ has zero-dimensional support, so assume $|\sigma| \geq 2$. There is an inclusion
  $$0 \to N_{F_{\sigma}} \xrightarrow{\alpha} N_{H_{\sigma}|_{C}},$$
  \cite[Lemma 2.6]{GHS}. The cokernel of $\alpha$ has zero dimensional support, so $H^1(C,N_{F_{\sigma}}(-T))=0$ implies $H^1(D,N_{H_{\sigma}|_{C}}(-T))=0$. By the Mayer-Vietoris sequence, it suffices that $H^1(N_{H_{\sigma}|_{R_i}}(-2))=0$ for each unstable component $R_i$ of $D$. It is enough to show $H^1(N_{\Delta_i}(-2))=0$ where $\Delta_i:=pr_{\sigma} \circ \beta_i \; : \; \PP^1 \simeq R_i \to (\PP^1)^{|\sigma|}$. This follows from the short exact sequence
  $$0 \to \mathcal{O}_{\PP^1}(2) \xrightarrow{d\Delta_i} \mathcal{O}_{\PP^1}(2)^{\oplus |\sigma|} \to N_{\Delta_i} \to 0.$$
\end{proof}
We end this section by proving a statement analogous to Proposition \ref{key-const-succ-unobst} with regards to the property of being geometrically in general position.
\begin{prop} \label{key-const-geo-general}
With notation as above, assume that $z_1, \ldots, z_{g-1-k}$ are distinct points of $C$ such that $h^0(C,f_i^*\mathcal{O}_{\PP^1}(1)(\sum_{j=1}^{g-1-k}z_j))=2$ for $1 \leq i \leq m$. Assume that $z_j \notin \{ x_i,y_i \; | \; 1 \leq i \leq n \}$ for all $1 \leq j \leq g-1-k$ and that $\{ f_1, \ldots, f_m \}$ are geometrically in general position with respect to $\{ z_i \}$. Then $\{ h_1, \ldots, h_m \}$ are geometrically in general position with respect to $\{ z_i \}$.
\end{prop}
\begin{proof}
Using $$0 \to \bigoplus_j \mathcal{O}_{R_j}(-2) \to \mathcal{O}_D \to \mathcal{O}_C \to 0$$ and twisting by $h_i^*\mathcal{O}_{\PP^1}(1)(\sum_{\ell} z_{\ell})$, we see $H^0(D,h^*_i\mathcal{O}_{\PP^1}(1)(\sum_{\ell} z_{\ell})) \simeq H^0(C,f^*_i\mathcal{O}_{\PP^1}(1)(\sum_{\ell} z_{\ell}))$. In particular, $h^0(h_i^*\mathcal{O}_{\PP^1}(1)(\sum_{\ell} z_{\ell}))=2$. From
$$0 \to \mathcal{O}_C(-\sum_i x_i+y_i) \to \mathcal{O}_D \to \bigoplus_i \mathcal{O}_{R_i} \to 0$$
we obtain $H^0(C,\omega_C \otimes f_i^*\mathcal{O}_{\PP^1}(-1)(-\sum_{\ell} z_{\ell})) \simeq H^0(D,\omega_D \otimes h_i^*\mathcal{O}_{\PP^1}(-1)(-\sum_{\ell} z_{\ell}))$. In particular, each $s \in H^0(D,\omega_D \otimes h_i^*\mathcal{O}_{\PP^1}(-1)(-\sum_{\ell} z_{\ell}))$ vanishes on each $R_j$. Hence if $u \in H^0(h_i^*\mathcal{O}_{\PP^1}(1))$, $$us \in H^0(D,\omega_D(-\sum_{\ell} z_{\ell}))$$ vanishes on each $R_j$ and hence lies in the subspace $H^0(C,\omega_C(-\sum_{\ell} z_{\ell}))$. From the determinantal description of the quadrics associated to pencils, $Q_{h_i}$ is a cone over $Q_{f_i}$. Hence $\{ h_1, \ldots, h_m \}$ are in geometrically general position with respect to $\{ z_i \}$.
\end{proof}

\section{Proof of Theorem \ref{geo-thm}}
\subsection{Outline of the proof} \label{outline}
We first provide an outline of the proof. Suppose $C$ has genus $g$, gonality $k$ and $m$ minimal pencils $f_i: C \to \PP^1$ satisfying the hypotheses of Theorem \ref{geo-thm}. Let $H=(h_i): D \to (\PP^1)^m$ be as in the Key Construction for $n=g+1-2k$, with $(x_i,y_i)$ general for $1 \leq i \leq n$, see \S \ref{sect=key-const}. Thus
$$D=C \bigcup_{j=1}^nR_i, \; \;R_i \simeq \PP^1$$
with ${h_i}_{|_C} \simeq f_i$ and $\deg ({h_i}_{|_{R_j}})=1$. We have $b_{g-k,1}(C,\omega_C) \geq m(g-k)$ as the pencils $\{f_1, \ldots, f_m \}$ are geometrically in general position by Proposition \ref{injectivity-syzy-scrolls}, so it suffices to prove
$b_{g-k,1}(C,\omega_C) \leq m(g-k)$.
We have an injective map
$$K_{g-k,1}(C, \omega_C) \hookrightarrow K_{g-k,1}(D,\omega_D)$$
on syzygy spaces. It suffices to prove $b_{g-k,1}(D,\omega_D)\leq m(g-k)$ or that
$$H=(h_i): D \to (\PP^1)^m \in \widetilde{\mathcal{H}}(m)$$
does not lie in the degeneracy locus $\text{Deg}(\mathcal{F}[m])$ of the morphism
$ \mathcal{F}[m]: \mathcal{V}[m] \to \mathcal{W}[m]$ from Proposition \ref{loc-free-crit}. But $\text{Deg}(\mathcal{F}[m])$ has codimension at most one and the dimension counts of \S \ref{sect=key-const}, show that, if $[H] \in \text{Deg}(\mathcal{F}[m])$, then there is a one-dimensional family
$$H_t=(h_{i,t}) \; : \; \mathcal{D}_t \to (\PP^1)^m \in \text{Deg}(\mathcal{F}[m])$$
with $H_0=H$ and $\mathcal{D}_t$ \emph{irreducible} for $t \neq 0$. For $t$ general, $H_t$ is successively unobstructed and $\{h_{i,t} \}$ is geometrically in general position for a general divisor of degree $g-1-k$. On the other hand, the forgetful map
$$\pi_{k+n} \; : \; \cM_{g+n,k+n}(\PP^1, \{0,1,\infty \} ) \to \mm_{g+n,3}$$
cannot be self-transverse with fibre of degree $m$ at $[\mathcal{D}_t, (p_t,q_t,r_t)]$, for general points $p_t,q_t,r_t$. Indeed, otherwise Proposition \ref{self-trans-order} together with \cite{hirsch} implies $b_{g-k,1}(\mathcal{D}_t,\omega_{\mathcal{D}_t}) \leq m(g-k)$, contradicting $[H_t] \in \text{Deg}(\mathcal{F}(m))$. By Proposition \ref{self-trans-lemma}, either $\pi^{-1}_{k+n} ([\mathcal{D}_t])$ contains an $(m+1)^{th}$ pencil $h_{m+1,t}$ in addition to $h_{1,t}, \ldots, h_{m,t}$, or there is some $i \leq m$ such that
$\dim |h^*_{i,t}\mathcal{O}_{\PP^1}(2)| \geq 3$. Furthermore, in the first case $$\text{lim}_{t \to 0} h_{m+1,t}=h_i$$ for some $1 \leq i \leq m$, by Proposition \ref{prop-hi}. In this case, set $L_t=h_{i,t}^*\mathcal{O}_{\PP^1}(1) \otimes h_{m+1,t}^*\mathcal{O}_{\PP^1}(1)$ and in the second case set $L_t=h_{i,t}^*\mathcal{O}_{\PP^1}(2)$. Observe $h^0(\mathcal{D}_t, L_t)\geq 4$. 

Let $\mathcal{D} \to \Delta$ be a fibred surface over a curve with fibre over $t$ equal to $\mathcal{D}_t$, let $\Delta^* \seq \Delta$ be the locus $t \neq 0$ and $\mathcal{D}^*$ the restriction of the fibred surface to $\Delta^*$ and suppose we have a line bundle $\mathcal{L}^{o} \in \text{Pic}(\mathcal{D}^*)$ with $\mathcal{L}^{o}_t=L_t$. Assume that $\mathcal{L}^{o}$ can be extended to a line bundle $\mathcal{L} \in \text{Pic}(\mathcal{D})$ with $$\mathcal{L}_0 \simeq h_{i,t}^*\mathcal{O}_{\PP^1}(2).$$ Further suppose $\mathcal{D}$ is smooth. The components $R_i \seq D$ define Cartier divisors. Define 
$$ \widetilde{\mathcal{L}}:=\mathcal{L}(\sum_{j=1}^n R_j).$$
Assuming $\{ x_i,y_i \}:=R_i \cap C$ are chosen generally, one expects $h^0(D,\widetilde{\mathcal{L}}_0)=3$, contradicting $h^0(\mathcal{D}_t, \widetilde{L}_t)=h^0(\mathcal{D}_t, L_t)\geq 4,$ for $t \in \Delta$ general\footnote{If one omits the twist, the Mayer--Vietoris sequence gives $h^0(D,\mathcal{L}_0)=3+n$.}.

In the full proof below, we work in the more general set-up from \cite{lin-syz}, allowing us to bypass the assumptions that $\mathcal{D}$ is smooth and that $\mathcal{L}^{o}$ can be extended. 

\subsection{The proof in full}
Recall the space $V_n(m)$, \S \ref{sect=key-const}. A point corresponds to $\displaystyle f: C \to (\PP^1)^m,$
with marking $(x_i,y_i)$, $1 \leq i \leq n$, together with $n$ marked maps $\displaystyle \beta_i : (\PP^1,(x_i,y_i)) \to (\PP^1)^m$
such that $f$ and $\beta_1, \ldots, \beta_n$ glue to produce a map 
$$h: D \to (\PP^1)^m.$$ Performing this procedure in families, we have a morphism $$q_n(m): \; V_n(m) \to \mm_{2k-1,k} ((\PP^1)^m,\{ 0,1,\infty \}).$$ Recall from \S \ref{comp} the morphism
$$\mathcal{F}[m] \; : \; \mathcal{V}[m] \to \mathcal{W}[m]$$
of sheaves, defined on the open locus $\widetilde{\H}(m) \seq \mm_{2k-1,k} ((\PP^1)^m,\{ 0,1,\infty \})$. We let $\H^{lf} (m) \seq \widetilde{\H}(m)$ be the open set such that
$\mathcal{V}[m]$ and $\mathcal{W}[m]$ are both locally free of rank of rank $ (2k-2){2k-1 \choose k}-m(k-1) $, see Proposition \ref{loc-free-crit}. Let 
$$\mathcal{K}(m):=\text{Deg} \, \mathcal{F}[m] \cap \H^{lf}(m) \seq \H^{lf} (m)$$ be the degeneracy locus of $\mathcal{F}[m]_{|_{\H^{lf} (m)}}$, i.e.\ locus where $\mathcal{F}[m]$ is not of full rank.
Define the set
$$Z_n(m) := pr_1\left(q_n(m)^{-1} (\mathcal{K}(m))\right)$$
where $pr_1: V_n(m) \to U_n(m) \seq \mm_{2k-1-n,k-n}((\PP^1)^m,\{ 0,1,\infty \}; 2n)$ is the projection. 

Theorem \ref{geo-thm} follows from the following, stronger result. Recall the morphism
$$\pi_k: \; \cM_{g,k}(\PP^1,\{0,1,\infty\}) \to \mm_{g,3}.$$
\begin{thm} \label{strongest-thm}
Let $C$ be an integral, nodal curve of genus $2k-1-n$ and gonality $k-n$ for $k \geq 3$, $0 \leq n \leq 2k-5$. Suppose $$\pi_{k-n}^{-1}\left(C, (p,q,r)\right)=\{f_1, \ldots, f_m: C \to \PP^1 \}.$$
Let $p,q,r \in C$ be distinct points in the \'etale locus of each $f_i$, $1 \leq i \leq m$.

\noindent For $1 \leq i \leq n$, let $(x_i,y_i) \in C$ be distinct pairs of points which are in the \'etale locus of $f_j$ and with $f_j(x_i) \neq f_j(y_i)$ for $1 \leq j \leq m$.
Assume:
 \begin{enumerate}
\item $\pi^{-1}_k(\widetilde{B}_{(x_i,y_i)},(p,q,r))$ is zero-dimensional of cardinality $m$, where $\widetilde{B}_{(x_i,y_i)}$ is the curve obtained from $C$ by glueing $x_i$ to $y_i$ for $ 1 \leq i \leq n$.
\item For all $S \seq \{ x_j,y_j \; | \; 1 \leq j \leq n \}$ of cardinality at most $n$, $h^0(C,f_i^*\mathcal{O}_{\PP^1}(2)(\sum_{s \in S} s))=3$ for $1 \leq i \leq m$.
\item $\{ f_1, \ldots, f_m \}$ are infinitesimally in general position with respect to $\{p,q,r\}$.
\item The line bundles $\{ f_1^*\mathcal{O}_{\PP^1}(1), \ldots, f_m^*\mathcal{O}_{\PP^1}(1) \}$ are geometrically in general position with respect to a general effective divisor of degree $k-2$ on $C$.
\end{enumerate}
Then if $F=(f_i): \; C \to (\PP^1)^m$, $[F,(x_i,y_i)] \notin  Z_n(m)$.
\end{thm}
\begin{proof}[Proof of Theorem \ref{geo-thm} assuming Theorem \ref{strongest-thm}]
Let $C$ be a smooth curve of genus $2k-1-n$ and non-maximal gonality $k-n$, $0 \leq n \leq 2k-5$, satisfying the assumptions of Theorem \ref{geo-thm}. In particular, assumptions $(3),(4)$ hold by hypothesis. For $1 \leq i \leq n$, let $(x_i,y_i) \in C$ be general pairs of distinct points. We have $$\pi_{k-n}^{-1}\left(C, (p,q,r)\right)=\{f_1,\ldots, f_m: C \to \PP^1 \}.$$
By Proposition \ref{prop-hi}, assumption $(1)$ holds. Next, $h^0(f_i^*\mathcal{O}_{\PP^1}(2))=3$ implies $$h^0(\omega_C\otimes f_i^*\mathcal{O}_{\PP^1}(-2))=n+1,$$
by Riemann--Roch. Hence, if $\{ x_j,y_j \; | \; 1 \leq j \leq n \}$ are sufficiently general and $S \seq \{ x_j,y_j \; | \; 1 \leq j \leq n \}$ has cardinality at most $n$ (or even $n+1$),
$$h^0(\omega_C\otimes f_i^*\mathcal{O}_{\PP^1}(-2)(-\sum_{s \in S} s))=n+1-|S|.$$
By Riemann--Roch again, one thus sees that $(2)$ holds for $\{ x_j,y_j \}$ general.

Thus  if $F=(f_i): \; C \to (\PP^1)^m$, $[F,(x_i,y_i)] \notin  Z_n(m)$. This means that, for any
$$\beta_i \; : \; (\PP^1,(x_i,y_i)) \to (\PP^1)^m, \; \; 1 \leq i \leq n$$
of degree one and satisfying $(\beta_i(x_i),\beta_i(y_i))=(F(x_i),F(y_i))$, the resulting map
$$[g\; : \; D \to (\PP^1)^m] \in \widetilde{\mathcal{H}}(m),$$
obtained by glueing $F$ and each $\beta_i$ as in Section \ref{sect=key-const} is not in $\mathcal{K}(m)$. Note that $[g] \in \H^{lf} (m)$ by Proposition \ref{loc-free-crit}, Proposition \ref{injectivity-syzy-scrolls} and Proposition \ref{key-const-geo-general}. Thus $b_{k-1,1}(D,\omega_D) = m(k-1)$ by Proposition \ref{loc-free-crit}. By \cite[Corollary 1]{V1}, $b_{k-1,1}(C,\omega_C) \leq b_{k-1,1}(D,\omega_D)=m(k-1)$. By Proposition \ref{injectivity-syzy-scrolls}, $b_{k-1,1}(C,\omega_C) \geq m(k-1)$ and
$$j \; : \; \bigoplus_{i=1}^m K_{g-k,1}(X_{f_i},\mathcal{O}_{X_{f_i}}(1)) \to K_{g-k,1}(C,\omega_C)$$
is injective. So we must have $b_{k-1,1}(C,\omega_C) = m(k-1)$ and $j$ is an isomorphism.
\end{proof}

We will prove Theorem \ref{strongest-thm} by induction on $n$. We start with the base case.
\begin{lem}
Theorem \ref{strongest-thm} holds in the case $n=0$ (and arbitrary $k\geq 3$).
\end{lem}
\begin{proof}
Note $[F=(f_i)] \in \H^{lf} (m)$ by Proposition \ref{loc-free-crit} and Proposition \ref{injectivity-syzy-scrolls}. Further, $\pi_k$ is self-transverse near each $f_i$ by Proposition \ref{self-trans-lemma}. Thus the local equation for the Hurwitz divisor $\mathfrak{hur}$ vanishes to order precisely $m$ at $[C]$, by Proposition \ref{self-trans-order}. It follows that the local equation for the Syzygy divisor $\mathfrak{Syz}$ vanishes to order precisely $m(k-1)$ and thus $b_{k-1,1}(C,\omega_C) \leq m(k-1)$, \cite{hirsch}, \cite[Thm.\ 3.1]{lin-syz}. Thus 
$[F] \notin \mathcal{K}(m)$ as required.
\end{proof}

It remains to prove the induction step. Let $1 \leq \ell \leq 2k-5$ and let $C$ be a smooth curve of genus $2k-1-\ell$ and gonality $k-\ell$, together with points $(p,q,r)$, $\{ x_i, y_i \; | \; 1 \leq i \leq \ell \}$ satisfying all the hypotheses of Theorem \ref{strongest-thm} (for $n=\ell$). Assume Theorem \ref{strongest-thm} holds with $n= \ell-1$ and arbitrary, fixed $k$. 

The key technical tool is the following proposition.
\begin{prop}\cite[Prop.\ 6.2, 6.3]{lin-syz} \label{twisting-prop}
Let $(\Delta,0)$ be an irreducible pointed variety and
$$\mathcal{G}: \mathcal{B} \to \PP^1_{\Delta} $$
a family of stable maps of genus $g$ and degree $k$ with $\mathcal{B}_{t}$ irreducible for $t \in \Delta$ general and
$$\mathcal{B}_0 \simeq C \cup R, \; \; R \simeq \PP^1, \; \; R \cap C=\{u,v\}, \; \; \deg \mathcal{G}_{R}=1$$
for irreducible $C$. Then, after a base change, there is a birational morphism $\nu: \widetilde{\mathcal{B}} \to \mathcal{B}$ between families of nodal curves of $\Delta$, together with a line bundle $\tau \in \text{Pic}(\widetilde{\mathcal{B}})$ such that
$$\tau_{t} \simeq \mathcal{O}_{\mathcal{B}_t}$$
for $ t \in \Delta$ general. Further, one of the following cases occur at $t=0$:
\begin{enumerate}
\item $\widetilde{\mathcal{B}}_0 \simeq \mathcal{B}_0$ and further ${\tau_{0}}_{|_C} \simeq \mathcal{O}_C(u+v)$, $\deg{\tau_{0}}_{|_R}=-2$.
\item $\widetilde{\mathcal{B}}_0$ is a blow-up of $\mathcal{B}_0$ at a node $p \in \{u,v\}$ with exceptional component $E$. Identifying $R$,$C$ with their strict transforms, ${\tau_{0}}_{|_C} \simeq \mathcal{O}_C(u+v)$, $\deg{\tau_{0}}_{|_R}=\deg{\tau_{0}}_{|_E}=-1$.
\end{enumerate}
Furthermore, additionally assume $h^0(\mathcal{G}^*_0\mathcal{O}_{\PP^1}(1))=2$, $\omega_C \otimes {\mathcal{G}_{0}}^*_{|_C} \mathcal{O}_{\PP^1}(-1)$
is base-point free and that the locus of $t \in \Delta$ with $\mathcal{B}_t$ reducible has codimension two about $0 \in \Delta$. Then, after a further base change, we may assume that we are in case $(2)$.
\end{prop}

We wish to prove that if $\pi^{-1}_{k-n}(C,(p,q,r))=\{f_1, \ldots, f_m \}$ and $F=(f_i) :  C \to (\PP^1)^m$ then $[F] \notin Z_{\ell}(m)$. We first prove a weakening of the induction step.
\begin{prop} \label{induction-weaker}
Let $1 \leq \ell \leq 2k-5$ and assume Theorem \ref{strongest-thm} holds for $n= \ell-1$. With notation as above, further assume:
\begin{align*}(2)': &\text{ For all $S \seq \{ x_j,y_j \; | \; 1 \leq j \leq \ell \}$ of cardinality at most $\ell+1$,} \\
&\text{$h^0(C,f_i^*\mathcal{O}_{\PP^1}(2)(\sum_{s \in S} s))=3$ for $1 \leq i \leq m$}.\end{align*}
Then $[F] \notin Z_{\ell}(m)$.
\end{prop}
The proof of Proposition \ref{induction-weaker} is very similar to the proof of \cite[Prop.\ 6.5]{lin-syz}. By Riemann--Roch again, $(2)'$ holds for $\{ x_j,y_j \}$ general.
\begin{proof}[Proof of Proposition \ref{induction-weaker}]
Suppose $[F: C \to (\PP^1)^m,(x_i,y_i)_{i \leq \ell} ] \in Z_{\ell}(m)$. Then there exist $\ell$ two-marked maps
$\beta_i: (\PP^1,(x_i,y_i)) \to (\PP^1)^m$ which glue to produce a genus $2k-1$, degree $k$ stable map
$$h: D \to (\PP^1)^m$$
such that $[h] \in \mathcal{K}(m)$. By Propositions \ref{key-const-geo-general} and \ref{loc-free-crit}, this is equivalent to having $b_{k-1,1}(D,\omega_D)>m(k-1),$
which does not depend on the choice of maps $(\beta_i)$ (provided they glue to $F$). For a suitable choice of $(\beta_i)$ we may ensure that each $(x_i,y_i)$ is outside the ramification locus of $h_{\sigma}=pr_{\sigma} \circ h$ for each $\sigma \seq \{1, \ldots, m \}$ with $|\sigma| \geq 2$.

Let $$C_{\ell}:= C \cup R_{\ell}, \; \; R_{\ell} \simeq \PP^1, \; \; R_{\ell} \cap C=\{x_{\ell},y_{\ell} \}$$
and let $F_{\ell}: C_{\ell} \to (\PP^1)^m$ be obtained by glueing $F$ and $\beta_{\ell}$. Obviously $[F_{\ell}, (x_i,y_i)_{i \leq \ell-1} ] \in Z_{\ell-1}(m)$. Any component 
$J \seq q_{\ell-1}(m)^{-1}(\mathcal{K}(m))$ has dimension at least 
$$\dim V_{\ell-1}(m)-1 \geq 6k-2\ell-m-2$$
by Lemma \ref{dim-V}. On the other hand, $V_{\ell}(m)$ has dimension $6k-2\ell-m-3$ by Lemma \ref{dim-V}, Corollary \ref{cor-tgt-space-base} and Proposition \ref{key-const-succ-unobst}. Thus the general point $$[F': C' \to (\PP^1)^m, (\beta'_i)] \in J$$
has irreducible base $C'$.

Thus we have a one dimensional, pointed variety $(\Delta,0)$ and marked families $$\mathcal{G}: \mathcal{B} \to (\PP^1_{\Delta})^m, \; \; (\mathbf{b}_i: \PP^1_{\Delta} \to (\PP^1_{\Delta})^m)_{i \leq \ell-1}  $$
with $[\mathcal{G}_t, (\mathbf{b}_{i,t})] \in J$ for all $t$, $(\mathcal{G}_0, (\mathbf{b}_{i,0}))=(F_{\ell}, (\beta_i)_{i \leq \ell-1})$ and with $\mathcal{B}_t$ irreducible for general $t$. By semicontinuity, $\mathcal{G}_t$ is successively unobstructed and, if 
$$g_{t,i}:=pr_i \circ \mathcal{G}_t,$$
then $\{ g_{t,i}^*\mathcal{O}_{\PP^1}(1) \; | 1 \leq i \leq m\}$ are geometrically in general position for general $t$. By \cite[Lemma 6.6 (I)]{lin-syz},  $h^0(g_{t,i}^*\mathcal{O}_{\PP^1}(2)(\sum_{s \in S} s))=3$ for $1 \leq i \leq m$ and any $S \seq \{ x_{t,j},y_{t,j} \; | \; 1 \leq j \leq \ell-1 \}$ of cardinality at most $\ell-1$, where $(x_{t,j},y_{t,j})$ are the markings on $\mathcal{B}_t$. Lastly, by the same argument as in \cite[Proposition 6.10]{lin-syz}, $$\pi^{-1}_k( C'_{(x_{t,i},y_{t,i})}, (p',q',r' ))$$ is zero-dimensional of cardinality $m$, for $C':=\mathcal{B}_t$ and $t$ general, where $p',q',r'$ are the base points and where $C'_{(x_{t,i},y_{t,i})}$ is the curve obtained from $C'$ by glueing $x_{t,i}$ to $y_{t,i}$ for $ 1 \leq i \leq \ell-1$. By the induction hypothesis, $[\mathcal{G}_t] \notin Z_{\ell-1}(m)$, which is a contradiction.

\end{proof}
We can now use Proposition \ref{induction-weaker} to prove the full theorem.
\begin{proof}[Proof of Theorem  \ref{strongest-thm}]
We follow the proof of \cite[Thm.\ 3.7]{lin-syz}. Proceeding by induction, one argues as in Proposition \ref{induction-weaker}. Thanks to the conclusion of Proposition \ref{induction-weaker}, one now has the additional information that the locus of reducible curves in the component $J \seq q_{\ell-1}(m)^{-1}(\mathcal{K}(m))$ has codimension two near $[F_{\ell}, (x_i,y_i)_{i \leq \ell-1} ]$. As in \cite[Thm.\ 3.7]{lin-syz}, we then may use \cite[Lemma 6.6 (II)]{lin-syz} as well an identical argument to \cite[Proposition 6.8]{lin-syz}. 
\end{proof}

\section{The case of two pencils}
In this section, we consider the classically studied case of curves with two minimal pencils. We start by constructing such examples using K3 surfaces. We work with curves on K3 surfaces for the reason that it is possible to gain a rather complete understanding of the Brill--Noether theory for such curves due to the very special properties enjoyed by moduli spaces of stable sheaves on K3 surfaces, \cite{lazarsfeld-BNP}, \cite{aprodu-farkas}.

\subsection{K3 Surfaces and Curves with Two Minimal Pencils}
Consider the lattice $\Lambda=\mathbb{Z}[C]\oplus \mathbb{Z}[E_1] \oplus \mathbb{Z}[E_2]$ and intersection matrix
\[ \left( \begin{array}{ccc}
(C \cdot C)& (C \cdot E_1) & (C \cdot E_2)  \\
  (E_1 \cdot C) & (E_1 \cdot E_1) & (E_1 \cdot E_2) \\
   (E_2 \cdot C) & (E_2 \cdot E_1) & (E_2 \cdot E_2)
  
  \end{array} \right) = \left( \begin{array}{ccc}
2g-2 & k & k  \\
  k & 0 & 2 \\
   k & 2 & 0
  
  \end{array} \right) \]
  for fixed $k \geq 3$. Then $\lambda$ has discriminant $4(k^2+2-2g)$.
  \begin{prop} \label{lattice-k3-1}
  Assume $g \leq \frac{k^2}{2}$, $k \leq \frac{g+1}{2}$ and $k \geq 3$. There exists a K3 surface $X_{\Lambda}$ with $\text{Pic}(X_{\Lambda}) \simeq \Lambda$ and with $E_1+E_2$ big and nef. Further, $E_i$ are the classes of smooth elliptic curves for $i=1,2$. If we further assume $g<\frac{k^2}{2}$, then $|C|$ is base-point free for $i=1,2$.
  \end{prop}
  \begin{proof}
 Since $g \leq \frac{k^2}{2}$, $\text{Disc}(\Lambda) >0$. The lattice $\Lambda$ is not positive definite, since $(E_1-E_2)^2=-4$. Hence it is even with signature $(1,2)$ and there exists a K3 surface $X_{\Lambda}$ with $\text{Pic}(X_{\Lambda}) \simeq \Lambda$, \cite{morrison-large}. We may take $E_1+E_2$ to be big and nef by \cite[Cor.\ 8.2.9]{huybrechts-k3}.
 
 To show that $E_1$ is the class of a smooth elliptic curve, it suffices to show that $E_1$ is nef, \cite[Prop.\ 2.3.10]{huybrechts-k3}. Since $E_1+E_2$ is base point free, $(E_1 \cdot E_1+E_2)=2>0$, the class $E_1$ is effective. It suffices to show that there is no smooth rational curve $R$ with $(R \cdot E_1)<0$ and $E_1-R$ effective. Suppose such an $R$ exists. We have $(R)^2=-2$ and $0 \leq (E_1-R \cdot E_1+E_2) \leq 2$. Suppose firstly that $(R \cdot E_1)\leq -2$. Then
 $(E_1-R)^2=-2(E_1 \cdot R_i) -2 \geq 2$. Applying the Hodge Index Theorem \cite[Remark 1.2.2]{huybrechts-k3} to $E_1-R, E_1+E_2$ then leads to a contradiction. So we must have $(R \cdot E_1)=-1$. Suppose $R=aC+bE_1+cE_2$, $a,b,c \in \mathbb{Z}$. Then $-1=(R \cdot E_1)=ak+2c$. Thus 
 $$(R-bE_1)^2=(aC+cE_2)^2=a^2(2g-2)+2c(ak)=a^2(2g-2-k^2)-ak \leq -2a^2-ak.$$
 On the other hand $(R \cdot E_1+E_2) \geq 0$ so $ak+2b=(R \cdot E_2) \geq 1$ and
 $$(R-bE_1)^2=-2+2b \geq -1-ak.$$ This forces $a=0$ and hence $-1=2c$ which is impossible. Thus $E_1$ is the class of a smooth elliptic curve and likewise for $E_2$.
 
 Assume $g<\frac{k^2}{2}$. We claim that $|C|$ is base-point free. We firstly claim that $C$ is nef. If $C$ is not nef, then there exists a smooth rational curve $R$ with $(C \cdot R)<0$ and $C-R$ effective. Write $R=aC+b_1E_1+b_2E_2$. Then $2b_i \geq -ak$ for $i=1,2$ since $E_1, E_2$ are nef and $R$ effective. Thus $$(R \cdot C)=a(2g-2)+b_1k+b_2k \geq a(2g-2-k^2).$$ As $2g-k^2 \leq 0$, this implies $a>0$. Next, since $C-R$ is effective, $2b_i \leq (1-a)k$ for $i=1,2$ and in particular $b_i \leq 0$. Now
 $$-2=(R)^2=a(C \cdot R)+b_1(E_1 \cdot R) + b_2(E_2 \cdot R)$$
 and all three terms on the right hand side of the above equation are non-positive. Thus $1 \leq a \leq 2$. If $a=2$, then $2b_i \leq -k$ and thus $b_i \neq 0$ for $i=1,2$. We then must have $(E_1 \cdot R)=(E_2 \cdot R)=0$ and so $b_i=-k$ for $i=1,2$. But $$(2C-kE_1-kE_2)^2=4(2g-2)-4k^2<-8,$$
 which is a contradiction. So we have $a=1$. If $(E_i \cdot R)=0$ then $b_j=-\frac{k}{2} \leq -2$ for $j \neq i \in \{1,2\}$ and then we are forced to have $(E_j \cdot R)=0$ and $b_i=-\frac{k}{2}$. We clearly cannot have $b_1=b_2=0$ nor $R=C-E_i$ for some $i=1,2$, so we must have $b_1=b_2=-\frac{k}{2}$. Then one computes
 $$(R)^2=(C-\frac{k}{2}(E_1+E_2))^2=2g-2-k^2<-2,$$
 since $g<\frac{k^2}{2}$, which is a contradiction.
 
 It remains to show that $|C|$ is base-point free. If $|C|$ is not base-point free, then, by Mayer's Theorem \cite{mayer} there exists a smooth elliptic curve $F$ and a smooth rational curve $\Gamma$ with $(F \cdot \Gamma)=1$ and $C=gF+\Gamma$, see also \cite[Cor 2.3.15]{huybrechts-k3}. As $E_i$ is nef and $(C \cdot E_i)=k <g$, we must have $(E_i \cdot F)=0$ for $i=1,2$. This forces $E_1=F=E_2$ which is impossible.
 
  \end{proof}
  In the next few lemmas we check that curves in the class $[C]$ have gonality $k$ and precisely two minimal pencils, both of which are of type $I$ (under certain bounds).
   \begin{lem} \label{MNbundles}
  Assume $g \leq \frac{4k^2}{9}$, $k \leq \frac{g+1}{2}$ and $k \geq 6$ and let $X_{\Lambda}$ be as in Proposition \ref{lattice-k3-1} and let $C \in [C]$ be a smooth curve. Let $M$ be a line bundle on $X_{\Lambda}$ such that, if $N=C-M$, then $h^0(M)=h^0(M_C)\geq 2$, $h^0(N)=h^0(N_C) \geq 2$, $h^1(M)=h^1(N)=0$ and $(M \cdot N)=\text{Cliff}(C)+2$. 
  Then either $M=E_i$ or $M=C-E_i$ for some $i \in \{1,2\}$.
  \end{lem}
  \begin{proof}
    Firstly, the assumption $g \leq \frac{4k^2}{9}$ and $k \geq 5$ implies $g<\frac{k^2}{2}$ as needed for Proposition \ref{lattice-k3-1}. We have $\mathcal{O}_C(E_i) \in W^1_k(C)$, so $\text{gon}(C) \leq k$ and hence $\text{Cliff}(C) \leq k-2$ and $$(M \cdot N) \leq k.$$  Let $C \in [C]$ be a smooth curve. Riemann--Roch implies that $(M)^2,(N)^2 \geq 0$. Write $M=aC+b_1E_1+b_2E_2$ for integers $a,b_1,b_2$. After interchanging $M$ and $N$, we may assume $a \leq 0$. We need to show $M=E_i$ for some $i \in \{1,2\}$.

 As both $M$ and $N$ are effective, by intersecting with $E_1,E_2$ we have $-ak \leq 2b_i \leq (1-a)k$ for $i=1,2$. We will firstly show we must have $a=0$. Suppose $a<0$. We have
  $$(M)^2=a^2(2g-2)+2ak(b_1+b_2)+4b_1b_2 \leq a^2(2g-2)+2ak(b_1+b_2)+(b_1+b_2)^2.$$
  Now if $f(t)=a^2(2g-2)+(2ak)t+t^2$ then the zeroes of $f(t)$ are $-ak \pm a\beta$ for $\beta:=\sqrt{k^2-2g+2}$. As $b_1+b_2 \geq -ak \geq -ak+a\beta$, then from the assumption $(M)^2 \geq 0$ we must have $b_1+b_2 \geq -ak-a \beta$. Next
  $$(N)^2 \leq (1-a)^2(2g-2)-2(1-a)k(b_1+b_2)+(b_1+b_2)^2$$
  and, since $(N)^2 \geq 0$, we get $b_1+b_2 \leq (1-a)(k-\beta)$. So,
  $$-a(k+\beta) \leq b_1+b_2 \leq (1-a)(k-\beta), \; \; \; \beta:=\sqrt{k^2-2g+2}.$$ In particular, $(1-a)(k-\beta) \geq -a(k+\beta)$ or $\beta \leq \frac{k}{1-2a}$. Since we are assuming $a \leq -1$, we have $$\beta \leq \frac{k}{3},$$ which gives $\frac{4k^2}{9}-g+1 \leq 0$. This contradicts our assumption $g \leq \frac{4k^2}{9}$.
  
  Lastly, suppose $a=0$. Then
  $$(M \cdot N)=(b_1+b_2)k-4b_1b_2 \geq (b_1+b_2)(k-(b_1+b_2))$$
  and $0 \leq b_i \leq \frac{k}{2}$ for $i=1,2$ and further $b_1+b_2 \leq k-\beta$. If $b_1+b_2=1$, then we are done, so assume $b_1+b_2 \geq 2$. We have seen that the assumption $g \leq \frac{4k^2}{9}$ implies $\beta > \frac{k}{3}$ and so $$2 \leq b_1+b_2 < \frac{2k}{3}.$$ Set $h(t)=t(k-t)$. Then
  $$(M \cdot N)\geq h(b_1+b_2) \geq \text{min}(h(2),h(\frac{2k}{3}))=\text{min}(2(k-2),\frac{2k^2}{9}) \geq k+1,$$
  for $k \geq 6$, which is a contradiction.

  \end{proof}
  
  The previous lemma immediately lets us compute the Clifford index of $C \in [C]$.
 \begin{lem} \label{cindex-k3}
  Assume $g \leq \frac{4k^2}{9}$, $k \leq \frac{g+1}{2}$ and $k \geq 6$. Let $X_{\Lambda}$ be as in Proposition \ref{lattice-k3-1}. Every smooth curve of class $[C]$
  has Clifford index $k-2$.
  \end{lem}
  \begin{proof}
By \cite{green-lazarsfeld-special} and \cite{martens}, there is a line bundle $M$ on $X_{\Lambda}$ such that, if $N=C-M$, then $h^0(M)=h^0(M_C)\geq 2$, $h^0(N)=h^0(N_C) \geq 2$, $h^1(M)=h^1(N)=0$ and $(M \cdot N)=\text{Cliff}(C)+2$. Then by Lemma \ref{MNbundles}, $(M \cdot N)=(E_i \cdot C-E_i)=k$ as required.
  \end{proof}
  
  In particular, the above lemma implies that, for every smooth curve $C$ of class $[C]$, $\text{gon}(C)=k$.
   \begin{lem} \label{h1van-lem}
  Assume $g<\frac{k^2}{2}$, $k \leq \frac{g+1}{2}$ and $k \geq 6$ and let $X_{\Lambda}$ be as in Proposition \ref{lattice-k3-1}. Then $h^1(X_{\Lambda}, C-2E_i)=0$ for $i=1,2$.
  \end{lem}
  \begin{proof}
 It suffices to consider the case $i=1$. We have $(C-2E_1)^2=2g-2-4k \geq -4$. Note that $H^0(2E_1-C)=0$ since $E_1$ is nef. In case $k=\frac{g+1}{2}$, it suffices by Riemann--Roch to show that $C-2E_1$ is not effective, whereas in case $k=\frac{g}{2}$ it suffices to show that that $C-2E_1$ is the class of a smooth $-2$ curve. If $k \leq \frac{g}{2}-1$, it suffices to show that $C-2E_1$ is nef. Thus, in all cases, it suffices to show that there is no smooth rational curve $R$ with $(R \cdot C-2E_1)<0$ and $C-2E_1-R$ effective and non-trivial.
 
 Suppose such an $R=aC+b_1E_1+b_2E_2$ exists. Since $2b_i \geq -ak$ for $i=1,2$,
 $$(R \cdot C-2E_1)=a(2g-2-2k)+b_1k+b_2(k-4)\geq a(2g-2-k^2).$$
 As $2g-2-k^2< 0$ we have $a \geq 1$. Next, since $C-2E_1-R$ is effective, $2(b_1+2) \leq (1-a)k $ and $ 2b_2 \leq (1-a)k$. Rewrite the first inequality as $b_1+2a \leq (1-a)(\frac{k}{2}-2)$. Now
 $$-2=(R)^2=a(C-2E_1 \cdot R)+(b_1+2a)(E_1 \cdot R)+b_2 (E_2 \cdot R),$$
 and all three terms on the right are non-positive, so $1 \leq a \leq 2$. 
 
 If $a=2$, we need $(E_1 \cdot R)=(E_2 \cdot R)=0$ which forces $b_1=b_2=-k$ and then $(R)^2 < -8$, which is impossible. So $a=1$. If $(E_1 \cdot R)=0$, then $b_2=-\frac{k}{2} \leq -3$ for which in turn forces $(E_2 \cdot R)=0$ and $b_1=-\frac{k}{2}$. If $(E_2 \cdot R)=0$ then $b_1=-\frac{k}{2}$ and $b_2 \in \{ -\frac{k}{2}, \frac{1-k}{2} \}$. In all cases,  $(R)^2 <-2$ using $g<\frac{k^2}{2}$, which is a contradiction.
 
Thus $a=1$ and $(E_1 \cdot R) \neq 0$, $(E_2 \cdot R) \neq 0$. In this case, we must have either $b_1=-2$ or $b_2=0$. We cannot have $(b_1,b_2)=(-2,0)$ since $C-2E_1-R$ is assumed to be nontrivial. If $b_2=0$ then $(E_1 \cdot R)=k$ which contradicts the above formula for $(R)^2=-2$. So we must have $b_1=-2$ giving the contradiction $(E_2 \cdot R)=k-4 \geq 2$. 
 
  \end{proof}
  Putting everything together, we now obtain:
  \begin{prop} \label{bpf-k3}
Assume $g \leq \frac{4k^2}{9}$, $k \leq \frac{g+1}{2}$ and $k \geq 6$ and let $X_{\Lambda}$ be as in Proposition \ref{lattice-k3-1}. Let $C$ be a smooth curve of class $[C]$ and let $A_i ={E_i}_{|_C}$ for $i=1,2$. Then $A_1$ and $A_2$ are not isomorphic, $h^0(C,2A_1)=h^0(C,2A_2)=3$ and $W^1_k(C)=\{A_1,A_2\}$. Further, the map $g: C \to \PP^1 \times \PP^1$ induced by $|A_1| \times |A_2|$ is birational to its image. Lastly, $C$ has gonality $k$ and satisfies bpf-linear growth.
  \end{prop}
  \begin{proof}
  By Lemma \ref{h1van-lem} and the sequence $0 \to 2E_i-C \to 2E_i \to  2A_i \to 0$ we see $h^0(C,2A_i)=h^0(X_{\Lambda},2E_i)=3$ for $i=1,2$. Next, observe that if $M$ is any line bundle on $X_{\Lambda}$ with $h^1(M)=0$, $(M \cdot E_i)>0$, then from
  $$0 \to M \to M+E_i \to (M+E_i)_{|_{E_i}} \to 0,$$
  we have $h^1(M+E_i)=0$. Thus $h^1(C+E_1-E_2)=h^1((C-2E_2)+E_2+E_1)=0$. From
  $$0 \to E_2-E_1-C \to E_2-E_1 \to A_2 - A_1 \to 0$$
  we see $h^0(A_2 - A_1)=h^0(E_2-E_1)=0$, since $E_2$ is nef and $(E_2 \cdot E_2-E_1)<0$. Thus $A_1$ and $A_2$ are not isomorphic.
  
  Suppose $L \in W^1_k(C)$. Since $\text{gon}(C)=k$ from Lemma \ref{cindex-k3}, $h^0(C,L)=2$ and $L$ is base-point free. By \cite[\S 4]{donagi-morrison}, there exists a line bundle $M$ on $X_{\Lambda}$ such that, setting $N=C-M$, $h^0(M)=h^0(M_C) \geq 2$, $h^0(N)=h^0(N_C) \geq 2$, $h^1(M)=h^1(N)=0$, $\deg(M_C) \leq g-1$ and $(M \cdot N)=k$. Further, there is a reduced $Z_0 \in |L|$ with $Z_0 \seq M \cap C$. From the proof of Lemma \ref{MNbundles}, this implies $M=E_i$ for some $i=1,2$. As $Z_0 \seq E_i \cap C$, and $k=\deg(Z_0)=(E_i \cdot C)$, we must have $L=A_i$. Lastly, $C$ has Clifford index $k-2$ and gonality $k$ by Lemma \ref{cindex-k3} and satisfies bpf-linear growth by Lemma \ref{MNbundles} and Theorem \ref{bpf-k3-criterion}.
  
 It remains to show that the map $g: C \to \PP^1 \times \PP^1$ induced by $|A_1| \times |A_2|$ is birational to its image. We claim the morphism $h:X_{\Lambda}\to \PP^1 \times \PP^1$ induced by $|E_1| \times |E_2|$ is finite. Otherwise there would be some effective class $R$ with $(R)^2=-2$, $(R \cdot E_1)=(R \cdot E_2)=0$. Thus
 $R=aC-\frac{ak}{2}(E_1+E_2)$ and $-2=a^2(2g-2-k^2)$ which is impossible as $2g-2-k^2 \leq -3$. Thus $h$ is finite of degree two and $g=h_{|_C}$. Suppose $g$ has degree two, i.e.\ $h^{-1}(h(C))=C$. As $h_*[C]=(k,k)$, $[C] \in \mathbb{Z}[E_1] \oplus \mathbb{Z}[E_2] \seq \Lambda$, which is a contradiction. Thus $g$ is birational.
   \end{proof}

\subsection{Coppens' Construction and Bpf-Linear Growth}
Fix $g \geq 8$ and $k \geq 4$. A necessary condition for the existence of a curve of genus $g$ carrying two independent pencils of type $I$ is $g \leq (k-1)^2$. Coppens proved that this bound is sufficient, \cite{coppens}. 

Let $C$ be a smooth curve of genus $g-1 \geq 8$. Two pencils $f_1, f_2: C \to \PP^1$ are said to be independent if there exists no automorphism $i: C \to C$ such that $f_2=f_1 \circ i$. The pencil $f_i$ is further said to be of type $I$ if $h^0(f_i^*\mathcal{O}_{\PP^1}(2))=3$.

Assume $C$ is a sufficiently general curve with two independent pencils $f_1, f_2$ and genus $g\leq (k-1)^2$. Then there are points $p \neq q \in C$ such that $f_i(p)=f_i(q)$ for $i=1,2$. Thus if $D$ is the nodal curve of genus $g$ obtained by identifying $p$ and $q$, we have two pencils $f'_1,f'_2: D \to \PP^1$. If $L_i:=f'^*_i\mathcal{O}_{\PP^1}(1)$, one shows that $L_1,L_2$ are the unique rank one, torsion free sheaves on $D$ with degree less than or equal to $k$ and at least two sections. Thus $D$ has gonality $k$ and precisely two minimal pencils, and furthermore one checks that $D$ can be deformed to a smooth curve with precisely two independent minimal pencils of degree $k$, both of type $I$.

We will show that, if one in addition assumes $k \leq \frac{g+9}{4}$ and that $C$ satisfies bpf-linear growth, then the nodal curve $D$ of genus $g$ may be deformed to a smooth curve, with precisely two minimal pencils, satisfying bpf-linear growth. Let $\overline{\text{Pic}}^d(D)$ denote the compactified Jacobian of rank one, torsion free sheaves of degree $k$ on $D$ and let $$W^1_d(D):= \{ M \in \overline{\text{Pic}}^d(D) \; | \; h^0(M) \geq 2 \}$$
be the closed subset of those sheaves with at least two sections. 
\begin{lem} \label{types-cpt}
Let $C,D$ be as above. Assume $C$ satisfies bpf-linear growth. Assume $0 \leq n \leq g-2k$ and let $Z \seq W^1_{k+n}(D)$ be a component of $W^1_{k+n}(D)$ of dimension at least $n$. Then the general point of $Z$ is a line bundle $[M] \in Z$ with $M=L_i(T)$ for $L_i \in W^1_k(D)$ and $T \seq D$ a reduced divisor in the smooth locus. Further we have equality $\dim(Z)=n$. 

\end{lem}
\begin{proof}
We argue by induction on $n$, the statement holding for $n=0$ as $W^1_k(D)=\{L_1,L_2 \}$. Let $[M] \in Z$ be a general point. We firstly claim $M$ is locally free. Suppose otherwise. Then, letting $\nu: C \to D$ be the normalization morphism, we have $M=\nu_*N$ for some line bundle $N \in W^1_{k+n-1}(C)$. Thus $\dim W^1_{k+n-1}(C) \geq n$ contradicting that $C$ satisfies bpf-linear growth. 

Thus the general point $M$ is locally free. Let $x \in D$ be the node. We firstly claim that $x$ is not a base point of $M$, for $[M] \in Z$ general. Suppose otherwise. Then, setting
$$M':=\text{Ker}(M \to M_x),$$
where $M \to M_x$ is the evaluation morphism, we have that $[M'] \in Z'$, where $Z' \seq W^1_{k+n-1}(D)$ is an irreducible closed subset. Note that, if $M'$ is locally free then $\dim \text{Ext}^1(M_x,M')=1$ and otherwise $\dim \text{Ext}^1(M_x,M')=2$, \cite[Lemma 5.10]{kemeny-singular}. Thus, in the former case, $\dim(Z') \geq n$, which contradicts the induction hypothesis, and in the latter case $\dim(Z')\geq n-1$ and the general point of $Z'$ is not locally free, which also contradicts the induction hypothesis. 

Thus the node $x$ is not a base point of $M$. Hence, setting $$M':=\text{Im}(H^0(M) \otimes \mathcal{O}_D \to M)$$
to be the base-point free part of $M$, we have that $M'=M(-T_1)$ for some effective divisor $T_1$ of degree $t_1$. Assume firstly that $t_1 \geq 1$, that is, that $M$ is not base-point free. We have $[M'] \in Z' \seq W^1_{k+n-t_1}(D)$ where $Z'$ has dimension at least $n-t_1$ and thus, by induction, we must have $M'=L_i(T_2)$ for a general effective divisor $T_2$ of degree $n-t_1$ and $i \in \{1,2\}$. Further, we must have equality $\dim(Z')=n-t_1$ and so $T_1$ may to chosen to be a general effective divisor of degree $t_1$. This gives the claim.

We are left with the case where $M$ is a base-point free line bundle. But this case cannot occur. Indeed, applying $\nu^*$ would produce a component $Z' \seq G^{1,\mathrm{bpf}}_{k+n}(C)$ of dimension $n$, which is impossible as we are assuming that $n \leq (g-1)+1-2k$ and $C$ satisfies bpf-linear growth.
\end{proof}

Let $W^{1,\mathrm{lf}}_d(D) \seq W^1_d(D)$ denote the open locus of locally free sheaves which we endow with the scheme structure of a determinantal variety in the usual way, \cite[Ch.\ IV]{ACGH1}. 
\begin{lem} \label{lf-smooth}
Let $C,D$ be as above with $W^1_k(D)=\{L_1,L_2\}$. Let $T \seq D$ be a general, reduced divisor in the smooth locus of degree $0 \leq n \leq g+2-2k$. Then $W^{1,\mathrm{lf}}_{k+n}(D)$ is smooth of dimension $n$ at the point $[L_i(T)]$.
\end{lem}
\begin{proof}
As already remarked above, $L_i$ are of type $I$, i.e.\ $h^0(D,L_i^{\otimes 2})=3$ for $i=1,2$. Further $h^1(L_i)=g+1-k$ by Riemann--Roch, so if $T$ is general $h^1(L_i(T))=h^0(\omega_D\otimes L_i^*(-T))=g+1-k-n$ and thus $h^0(L_i(T))=2$ by Riemann--Roch. Likewise, $h^0(D,L_i^{\otimes 2}(T))=3$. By Proposition 4.2 of \cite[Ch.\ IV]{ACGH1} (which goes through verbatim in the case of an integral, nodal curve), the tangent space to $W^{1,lf}_{k+n}(D)$ is then
$(Im(\mu))^{\perp}$, where $$\mu: H^0(D,L_i(T)) \otimes H^0(D,\omega_D \otimes L_i^*(-T)) \to H^0(D,\omega_D)$$ is the Petri map. Thus $T_{[L_i(T)]}W^{1,\mathrm{lf}}_{k+n}(D)$ has dimension
$g-2(g+1-k-n)+\dim\text{Ker}(\mu)$. By the base-point free pencil trick \cite[Pg. 126]{ACGH1} (which holds in our context), $\dim \text{Ker}(\mu)=h^0(\omega_D\otimes L_i^{\otimes -2}(-T))=g+2-2k-n$, by Riemann--Roch. Thus $\dim T_{[L_i(T)]}W^{1,\mathrm{lf}}_{k+n}(D)=n$. But obviously $\dim W^{1,\mathrm{lf}}_{k+n}(D) \geq n$, so this finishes the proof.
\end{proof}
Putting the above lemmas together we obtain the following result.
\begin{prop} \label{bpf-range-quarter}
Let $C,D$ be as above where $D$ has genus $g$ and gonality $k$. Assume $C$ satisfies bpf-linear growth and let $3 \leq k \leq \frac{g+8}{4}$. Let $(\Delta,0)$ be a smooth, pointed curve and $\mathcal{D} \to \Delta$ a flat family of nodal curves such that $\mathcal{D}_0 \simeq D$ and $\mathcal{D}_t$ is smooth of gonality $k$ for $t \neq 0$. Assume $W^1_k(\mathcal{D}_t)=\{ L_{1,t}, L_{2,t} \}$, where $L_{1,t}$ and $L_{2,t}$ are not isomorphic and of type $I$. Then $\mathcal{D}_t$ satisfies bpf-linear growth for $t \in \Delta$ general.
\end{prop}
\begin{proof}
By Lemma \ref{bpf-last-red}, it suffices to check $\dim W^{1,\mathrm{bpf}}_{k+n}(\mathcal{D}_t)<n$ for $1 \leq n \leq g-2k$. Suppose $\dim W^{1,\mathrm{bpf}}_{k+n}(\mathcal{D}_t) \geq n$ for general $t$ and some  $1 \leq n \leq g-2k$. After a finite base change, we have a variety $\mathcal{W}$, together with a proper morphism $\mathcal{W} \to \Delta$, with fibre over $t$ equal to $W^{1}_{k+n}(\mathcal{D}_t)$, cf.\ \cite[Ch.\ XXI]{ACG2}. By assumption, we have a component $\mathcal{I}_{\mathrm{bpf}} \seq \mathcal{W}$ of relative dimension at least $n$ over $\Delta$ whose general point is base-point free. By Lemma \ref{types-cpt}, the fibre $\mathcal{I}_{\mathrm{bpf},0}$ over $0$ must contain the point $[L_i(T)]$ for some $i \in \{1,2\}$ and $T$ any general Cartier divisor of degree $n$. But, by Lemma \ref{lf-smooth}, $\mathcal{W}$ is smooth at $[L_i(T)]$. But there is a closed subset $J \seq \mathcal{W}$ of relative dimension $n$ containing all points of the form $L_{1,t}(T_t)$ and $L_{2,t}(T_t)$ for $T_t$ a general effective divisor of degree $n$ on $\mathcal{D}_t$. As $[L_i(T)]$ lies in a unique component, we must have $\mathcal{I}_{\mathrm{bpf}} \seq J$. But, for $t$ general, $h^0(L_{i,t})=2$ for $i=1,2$, and hence $h^1(\omega_{\mathcal{D}_t}\otimes L_{i,t}^*)=g-k+1$, and since $n \leq g-2k\leq g-k+1$, $h^1(\omega_{\mathcal{D}_t}\otimes L_{i,t}^*(-T_t))=h^1(\omega_{\mathcal{D}_t}\otimes L_{i,t}^*)-n$ which implies $h^0(L_{i,t}(T_t))=2$, so that $L_{i,t}(T_t)$ is not base-point free. This contradicts the assumption that the general point of $\mathcal{I}_{\mathrm{bpf}}$ is base-point free.
\end{proof}

Let $\cM_{g,k}(2) \seq \cM_g$ denote the moduli space of smooth curve of genus $g$ and gonality $k \leq \frac{g+1}{2}$ such that $ W^1_k(C)=\{L_1,L_2 \}$ where $L_1$ and $L_2$ are independent and of type $I$. Assume $k \geq 4$ and $g \geq 8$. Then $\cM_{g,k}(2)$ is nonempty if and only if $g \leq (k-1)^2$ and is irreducible, \cite{coppens}, \cite{Ty}.
\begin{thm}
Assume $k \geq 6$, $g \geq 8$ and let $[C] \in \cM_{g,k}(2)$ be general. Then $b_{g-k,1}(C,\omega_C)=2(g-k)$.
\end{thm}
\begin{proof}
We will firstly show that $[C]$ satisfies bpf-linear growth. Note that if $k \geq 5$ we always have either $g \leq \frac{4k^2}{9}$ or $k \leq \frac{g+8}{4}$. If $g \leq \frac{4k^2}{9}$, then the fact that a general $[C]$ satisfies bpf-linear grwoth follows from Proposition \ref{bpf-k3}. The remaining cases now follow by Coppens' inductive construction and Proposition \ref{bpf-range-quarter}. Note that the base case in Coppens' construction is $g=2k-1$, which falls into the range of Proposition \ref{bpf-k3} (which in particular gives a new proof of \cite[\S 2]{coppens}).

It follows from Coppens' construction that the two minimal pencils on a general point $[C]$ have only ordinary ramification. Let $[C] \in \cM_{g,k}(2)$ be general and choose general points $p,q,r \in C$. We claim that the two pencils $f_1, f_2: C \to \PP^1$ are infinitesimally in general position with respect to $\{p,q,r\}$. This amounts to showing $H^1(C,N_F(-p-q-r))=0$, where $F=(f_1,f_2)$. As in \cite{arbarello-cornalba}, there is an exact sequence
$$0 \to \mathcal{O}_Z \to N_F \to N'_F \to 0,$$
where $Z$ has zero-dimensional support and $N'_F$ is a line bundle. By \cite[Prop.\ 2.4]{arbarello-cornalba}, $N_F$ is a line bundle of degree $2g-2+4k$ and so $N_F(-p-q-r)$ has degree greater than $2g-1$, giving $H^1(C,N_F(-p-q-r))=0$. 

It remains to prove that $f_1,f_2$ are geometrically in general position. Let $T \in C_{g-1-k}$ be general. We need to show that $[Q_{f_1}] \neq [Q_{f_2}] \in |\mathcal{O}_{\PP^k}(2)|$, in the notation of Section \ref{GGP}. Assume otherwise. Then $\widetilde{Q}_{f_1}=\widetilde{Q}_{f_1}$ as quadrics in $\PP^{g-1}$ and hence their resolutions $\PP(V_1)$, $\PP(V_2)$ are isomorphic. Noting that $C$ does not entirely lie in the vertex of $\widetilde{Q}_{f_1}$ (as $C \seq \PP^{g-1}$ is nondegenerate), we have an isomorphism $\psi: \PP(V_1) \to \PP(V_2)$ preserving both the hyperplane class and the image of $C$. Let $R$ and $H$ denote the ruling and hyperplane class of $\PP(V_1)$, and let $R_2 \in \text{Pic}(\PP(V_1)) \simeq \text{Pic}(\PP(V_2))$ denote the class of the ruling. So $R_2=\alpha R+\beta H$ for $\alpha, \beta \in \mathbb{Z}$. Intersecting with $H^{\dim \PP(V_1)-1}$ produces the equation $1=\alpha+\deg(\PP(V_1))\beta=\alpha+2\beta$. Restricting to $C$, we obtain
$$L_2 \simeq (1-2\beta)L_1 +\beta \omega_C,$$
where $L_i \simeq f_i^*\mathcal{O}_{\PP^1}(1)$. Taking degrees and using that $k \neq g-1$, we see $\beta=0$ and so $L_1 \simeq L_2$ which is a contradiction.
\end{proof}

\appendix
\section{Bpf-Linear Growth} \label{k3}
In this Appendix we gather some results on bpf-linear growth which are implicit in the existing literature, in particular in works of Mumford, Keem and Aprodu--Farkas. We begin with an observation about the bpf-linear growth condition. For a smooth curve $C$, recall
$$W^r_d(C):=\{ L \in \text{Pic}^d{C} \; | \; h^0(C,L) \geq r+1 \}$$
which can be given the structure of a determinantal variety, \cite{ACGH1}. We set
$W^{r,\mathrm{bpf}}_k(C) \seq W^r_d(C)$ to be the open locus of base-point free line bundles.
\begin{lem}
A smooth curve $C$ satisfies bpf-linear growth if and only if $\dim W^1_k(C)=0$ and $\dim W^{1,\mathrm{bpf}}_{k+n}(C)<n$ for $1 \leq n \leq g+1-2k$.
\end{lem}
\begin{proof}
It is clear that the above condition is equivalent to \begin{align*}
\dim W^1_{k+m}(C) &\leq m,  \;\; \text{for $0 \leq m \leq g-2k+1$} \\
\dim W^{1,\mathrm{bpf}}_{k+m}(C) &<m, \;\; \text{for $0 < m \leq g-2k+1$} .\end{align*} From the proof of \cite[Lemma 3.3]{lin-syz}, this is equivalent to the bpf-linear growth conditions
 \begin{align*}
\dim G^1_{k+m}(C) &\leq m,  \;\; \text{for $0 \leq m \leq g-2k+1$} \\
\dim G^{1,\mathrm{bpf}}_{k+m}(C) &<m, \;\; \text{for $0 < m \leq g-2k+1$} .\end{align*} 
\end{proof}

The bpf-linear growth condition is implicit in the well-known work of Mumford and Keem on dimensions of Brill--Noether loci. In fact, we have:
\begin{thm} [Mumford--Keem] \label{mumford-keem}
Let $C$ be a smooth curve of genus $g$ and gonality $k \geq 3$.
\begin{enumerate} 
\item If $k=3$, then $C$ satisfies bpf-linear growth unless $C$ admits a degree two morphism to an elliptic curve.
\item If $k=4$ and $g \geq 11$, then $C$ satisfies bpf-linear growth unless $C$ admits a degree two morphism to a curve of genus $\ell$ for $1 \leq \ell \leq 2$.
\item If $k=5$ and $g \geq 15$ then $C$ satisfies bpf-linear growth unless $C$ admits either a degree three morphism to an elliptic curve or a degree two morphism to a curve of genus three.

\end{enumerate}
\end{thm}
\begin{proof}
The first part of the Theorem is proven in the course of the proof of the theorem from \cite[Appendix]{mumford-prym}, whereas the second and third parts are proven in \cite[Thm 2.1, 3.1]{Keem}.
\end{proof}
The arguments of Mumford and Keem also show that, provided $k$ is small enough, the condition $\dim W^{1,\mathrm{bpf}}_{g-k+1}(C)<g-2k+1$ of bpf-linear growth is redundant.
\begin{lem} \label{bpf-last-red}
Let $C$ be a smooth curve of gonality $3 \leq k \leq \frac{g+8}{4}$. Suppose $\dim W^{1,\mathrm{bpf}}_{k+n}(C)<n$ for $1 \leq n \leq g-2k$ and $\dim W^1_k(C)=0$. Then $\dim W^{1,\mathrm{bpf}}_{g-k+1}(C)<g-2k+1$.
\end{lem}
\begin{proof}
We follow an argument of Mumford and Keem. Suppose there is a component $Z \seq W^{1}_{g-k+1}(C)$ of dimension at least $g+2-2k$ such that the general element $[L] \in Z$ is base-point free. By the assumption  $\dim W^{1,\mathrm{bpf}}_{k+n}(C)<n$ for $1 \leq n \leq g-2k$ and $\dim W^1_k(C)=0$, we must have $h^0(L)=2$. By results of Kempf and Severi (see \cite[\S 1]{Keem}), this gives
$$\rho(g,1,g+1-k)+h^0(\omega_C \otimes L^{-1})=g-2k+h^0(\omega_C \otimes L^{-1}) \geq g+2-2k$$
and hence $h^0(\omega_C \otimes L^{-2}) \geq 2$. On the other hand, $\deg(\omega_C \otimes L^{-2})=2k-4$, so $\dim W^1_{2k-4}(C) \geq g+2-2k$. But this contradicts H.\ Martens' Theorem \cite{h-martens} since $3 \leq k \leq \frac{g+8}{4}$.
\end{proof}

Results of \cite{aprodu-farkas} provide a sufficient conditions for bpf-linear growth for curves on a K3 surface:
\begin{thm} \label{bpf-k3-criterion}
Let $C$ be a curve of gonality $k \leq \lfloor \frac{g+1}{2} \rfloor$ and genus $g \geq 3$, abstractly embedded on a K3 surface $S$. Assume $C$ has Clifford dimension one. Suppose that for any line bundle $M$ on $S$ satisfying the properties
\begin{enumerate}
\item $h^0(M)=h^0(M_C) \geq 2$ and $h^1(M)=0$,
\item Setting $N=C-M$, we have $h^0(N)=h^0(N_C) \geq 2$ and $h^1(N)=0$,
\item $(M \cdot N)\leq k$
\end{enumerate}
then either $M_C \in W^1_k(C)$ or $N_C \in W^1_k(C)$. Then $C$ satisfies bpf-linear growth.

\end{thm}
\begin{proof}
Set $L=\mathcal{O}_S(C)$. For any base point free pencil $A$ of degree $d$ on $C$, with $k < d \leq g+1-k$,  we have a corresponding short exact sequence
$$0 \to M \to E \to N \otimes I \to 0,$$ 
on $S$, where $E$ is a rank two Lazarsfeld--Mukai bundle satisfying $c_2(E)=d$, $L,M$ are line bundles and $I$ is the ideal sheaf of a zero-dimensional subscheme of $S$ of length $\ell \geq 0$, \cite[Lemma 4.4]{donagi-morrison}. We further may assume $M-N$ is effective (possibly after swapping $M$ and $N$ if $\ell=0$), $h^0(M), h^0(N)  \geq 2$ and $N$ is base-point free, see \cite[Lemma 2.1]{ciliberto-pareschi}.

From the results in \cite[Lemmas 3.9,3.10]{aprodu-farkas}, it suffices to show we have the strict inequality $(M \cdot N) > k$. This is automatic if $\text{Cliff}(M_C)>k-2$ or $h^1(S,M) \neq 0$, so assume $\text{Cliff}(M_C)=k-2$, $h^1(M)=0$ and $(M \cdot N) \leq k$.

From
$$0 \to N^* \to M \to M_C \to 0,$$
and the fact that $h^0(N^*)=0$, $h^1(N)=0$, we see $h^0(M)=h^0(M_C) \geq 2$. Further, $h^0(N)=h^1(M_C)=h^0(N_C) \geq 2$. Thus, by our assumptions, $M_C \in W^1_k(C)$ or $N_C \in W^1_k(C)$. As we are assuming $M-N$ is effective and since $C$ is nef, $(M-N \cdot C) \geq 0$ so $N \cdot C \leq M \cdot C$ and thus $N_C \in W^1_k(C)$. Thus $h^0(N)=2$, $h^1(N)=h^2(N)=0$ and the base-point free line bundle $N$ is the class of a smooth elliptic curve (by Riemann--Roch).

 But, by \cite[Corollary 4.5]{donagi-morrison}, there is some divisor $N' \in |N|$ and some reduced $A' \in |A|$ such that $A' \seq N' \cap C$. As $C$ is irreducible of genus greater than one, this forces $d \leq k=(N \cdot C)$ which is contradiction.
\end{proof}
As a Corollary, we obtain the claim that a sufficient condition for bpf-linear growth is that the only line bundles $A$ achieving the Clifford index are elements of $W^1_k(C)$ .
\begin{cor}
Let $C$ be a curve of gonality $k \leq \lfloor \frac{g+1}{2} \rfloor$, abstractly embedded on a K3 surface $S$.
Suppose that if $A$ is a line bundle with $h^0(A) \geq 2$, $h^1(A) \geq 2$, $\deg(A) \leq g-1$ and $\text{Cliff}(A)=\text{Cliff}(C)$, then we have $A \in W^1_k(C)$. Then $C$ satisfies bpf-linear growth.
\end{cor}
\begin{proof}
The same as the proof of Theorem \ref{bpf-k3-criterion}, with the assumptions forcing $N_C \in W^1_k(C)$. 
\end{proof}

\end{document}